\documentclass[preprint,aps]{revtex4-2}

\usepackage[utf8]{inputenc}
\usepackage{amsmath,amssymb}
\usepackage{amsthm}

\usepackage{algorithm}
\usepackage{algpseudocode}
\usepackage{listings}
\usepackage{longtable}
\usepackage{graphicx}
\usepackage{hyperref}

\newtheorem{Theorem}{Theorem}

\newtheorem{Proposition}{Proposition}

\newtheorem{Definition}{Definition}
\newtheorem{Remark}{Remark}

\begin{document}

\title{Delay Modeling with Conformable and Caputo Derivatives: Analytical and Computational Insights}


\title{Delay Modeling with Conformable and Caputo Derivatives: Analytical and Computational Insights}

\author{Yhon Flores}
\affiliation{Departamento de Matemáticas, Universidad Católica del Norte, Avenida Angamos 0610, Casilla 1280, Antofagasta, Chile}

\author{Michel Molina del Sol}
\affiliation{Departamento de Matemáticas, Universidad Católica del Norte, Avenida Angamos 0610, Casilla 1280, Antofagasta, Chile}
\author{Genly Leon}
\email{genly.leon@ucn.cl}
\affiliation{Departamento de Matemáticas, Universidad Católica del Norte, Avenida Angamos 0610, Casilla 1280, Antofagasta, Chile}
\affiliation{Department of Physics, Universidad de Antofagasta, 1240000 Antofagasta, Chile}
\affiliation{Institute of Systems Science, Durban University of
Technology, Durban 4000, South Africa}
\affiliation{Centre for Space Research, North-West University, Potchefstroom 2520, South Africa}

\author{Byron Droguett}
\affiliation{Department of Physics, Universidad de Antofagasta, 1240000 Antofagasta, Chile}

\author{Guillermo Fernández-Anaya}
\affiliation{Departamento de Física y Matemáticas, Universidad Iberoamericana Ciudad de México, México}

\author{Yoelsy Leyva}
\affiliation{Departamento de Física, Facultad de Ciencias, Universidad de Tarapaca, Casilla 7-D, Arica, Chile}

\begin{abstract}
This work presents an analytical and computational study of fractional-order delay differential equations formulated using both the conformable and Caputo derivatives. For the conformable case, we develop the associated integral, exponential function, and Laplace transform, showing how the conformable Laplace framework preserves algebraic structure and facilitates explicit solutions. Delay terms are treated through series expansions and transform-based methods, ensuring causal and finite representations. In parallel, Caputo-based formulations are examined, highlighting the challenges posed by convolutional memory kernels and the potential for long-term numerical instability. Numerical implementations are carried out using mesh-aligned algorithms: Euler and Runge--Kutta schemes for conformable dynamics, and Euler, L2--$\sigma$, and a series--anchored predictor--corrector method for Caputo dynamics. Comparative experiments demonstrate that conformable derivatives yield stable, consistent agreement between analytic and numerical solutions, whereas Caputo dynamics require higher-order or series-anchored schemes to suppress discretization noise and maintain long-term accuracy. These results underscore the advantages of the conformable formalism in modeling dynamic phenomena with delay and memory, offering a tractable and physically interpretable alternative to integral-based fractional models.
\end{abstract}

\maketitle

 
\section{Introduction}
\noindent

Traditional differential models often fall short when capturing power-law dynamics characterized by frequency dependence, long-range memory, and non-locality. Fractional calculus offers an elegant extension to classical calculus, enabling differentiation and integration to be performed on non-integer orders. This generalization accounts for historical system behavior, making it particularly suitable for applications in which past states influence present dynamics.

Fractional derivatives have become indispensable in modeling systems with complex temporal structure, including viscoelastic materials, electrical networks, and anomalous diffusion. Their utility spans diverse fields—from quantum field theory and gravity~\cite{Lim:2006hp, El-nabulsi:2013hsa}, to black hole physics~\cite{Vacaru:2010wn}, and cosmological studies~\cite{Moniz:2020emn, Garcia-Aspeitia:2022uxz, Gonzalez:2023who}. 
In these contexts, fractional frameworks enable refined observational tests and contribute to the derivation of viable equations of state~\cite{LeonTorres:2023ehd, Micolta-Riascos:2023mqo}.

Fractional extensions of Newtonian dynamics and cosmology have been widely investigated. Higher-order non-local variants of Newton’s laws~\cite{ELnABULSI201765} and fractional Friedmann--Robertson--Walker equations provide alternative descriptions of disordered or accelerated motion. Two main approaches dominate the field: (i) direct insertion of fractional operators into classical field equations~\cite{Barrientos:2020kfp}, and (ii) geometric formulations based on the Fractional Action-Like Variational Approach (FALVA)~\cite{El-nabulsi-Torres-2008}. Both have produced viable models consistent with late-time cosmic acceleration~\cite{LeonTorres:2023ehd, Rasouli2024}.

A broader overview is given in \cite{Marroquin:2024ddg}, where fractional calculus is presented as a framework for non-local and memory-dependent power-law phenomena. The authors review applications ranging from mechanics to cosmology. By constructing a fractional Einstein--Hilbert action with a non-minimally coupled scalar field $\xi R\phi^{2}$, they obtain modified cosmological equations and analyze their dynamics using asymptotic and dynamical-systems methods, showing that fractional-order gravity can reproduce late-time acceleration without dark energy.
Further progress appears in the fractional Einstein--Gauss--Bonnet scalar field cosmology of \cite{Micolta-Riascos:2024vbm}. Fractional calculus is used to modify the action, yielding fractional Friedmann and Klein--Gordon equations depending on the order $\mu$. Exact solutions with exponential potentials, Gauss--Bonnet couplings, and logarithmic scalar fields arise, governed by parameters $m$, $\alpha_{0}=t_{0}H_{0}$, and $\mu$. A dynamical-systems analysis shows that the Gauss--Bonnet coupling can mimic dark-sector behavior. Using cosmic chronometers, SNe Ia, black hole shadows, and strong lensing, the authors find $\alpha_{0}=1.38\pm0.05$, $m=1.44\pm0.05$, and $\mu=1.48\pm0.17$, consistent with $\Omega_{m,0}=0.311\pm0.016$ and $h=0.712\pm0.007$, implying $t_{0}=19.0\pm0.7$ Gyr. These results support the acceleration of power-law solutions and reinforce fractional calculus as a promising alternative to $\Lambda$CDM.

Recent advances in fractional calculus have highlighted the importance of incorporating memory and delay effects into dynamical models. In particular, fractional time-delay differential equations provide a powerful framework for describing systems in which past states influence the present evolution. These ideas have found compelling applications in cosmology, where delayed fractional dynamics can capture phenomena such as bulk viscosity and nonlocal interactions in the early universe. Micolta‑Riascos et al. (2025) \cite{retard} developed analytic and numerical techniques for fractional time‑delayed equations, proposed Laplace‑domain formulations, and demonstrated their relevance to cosmological studies, thereby establishing a rigorous foundation for memory‑aware modeling in gravitational contexts. 

Within the spectrum of non-integer derivatives, the conformable derivative introduced by Khalil~\cite{Khalil14} stands out for maintaining compatibility with classical differential rules—such as linearity and the product rule—thus simplifying its integration into existing differential frameworks. Defined as a deformation of the standard derivative for orders $\alpha \in (0,1]$, it offers computational tractability and intuitive structure, especially for systems requiring memory effects without full convolution-based formalism.

Recent advances have extended conformable calculus to delay differential equations (DDEs), which naturally incorporate history-dependent behavior. In population dynamics and control theory, such models demand tailored analytical tools, including Laplace-based solution strategies and adapted stability criteria. The conformable approach provides a clean, closed-form alternative to convolution-based fractional-delay models, preserving causal structure and simplifying the handling of initial conditions. Furthermore, Mohammadnezhad, Eslami, and Rezazadeh~\cite{Mohammadnezhad22} have made significant contributions in this area, providing sufficient conditions for asymptotic and exponential stability in nonlinear systems with delay under conformable derivatives, thus strengthening the applicability of the conformable framework in practical scenarios. 
    
The broader landscape of time-delayed differential systems has seen significant growth. Studies on summable dichotomies \cite{del2012bounded} and Volterra-type equations highlight the existence of bounded, nearly periodic solutions~\cite{del2018bounded,del2011almost}. Weighted exponential trichotomy provides insights into long-term dynamics for linear and nonlinear systems~\cite{cuevas2010weighted, vidal2008weighted}, while asymptotic expansions remain vital for systems with infinite delays~\cite{cuevas2009asymptotic}.

For nonlinear fractional DDEs, solution existence and stability remain central concerns~\cite{geremew2024existence}. Langevin-type equations with fractional delays are now leveraged in engineering applications~\cite{huseynov2021class}, supported by numerical methods including finite difference schemes~\cite{moghaddam2013numerical}, spectral collocation~\cite{dabiri2018numerical}, and computational algorithms~\cite{amin2021computational}. Optimal control using Dickson polynomials~\cite{chen2021optimal} and stabilization strategies~\cite{lazarevic2011stability, pakzad2013stability} and synchronization~\cite{luo2011complex} further enhance the practical relevance of this field.

Oscillatory systems governed by time-delayed fractional derivatives—relevant in electronics, mechanics, and biological modeling—are addressed through comprehensive stability frameworks~\cite{leung2013fractional, hu2016stability}. These studies underscore the need for rigorous mathematical tools that account for memory and delay, particularly in high-precision applications.

In this work, we investigate a class of delay fractional differential equations governed by conformable derivatives. We establish rigorous analytic foundations for their behavior, introduce a Laplace-domain formulation tailored to delayed terms, and develop solution techniques that are adaptable to both theoretical analysis and physical applications. We aim to provide a mathematically consistent, memory-aware framework that addresses the central challenges of fractional-delay modeling. In addition, we examine analogous problems formulated with Caputo derivatives. While both approaches represent distinct generalizations of the same classical integer-order problem, our findings indicate that Khalil’s conformable derivative framework yields more effective replacements in practice.

\noindent
Several appendices based on~\cite{mathai2008special,kiryakova2010special,yang2021theory,agarwal2020special} are included to ensure the study's self-contained nature.

\section{Preliminaries} \label{sec:preliminares}

Among formulations of fractional derivatives—Riemann–Liouville, Caputo, Grünwald–Letnikov, and Hadamard—an alternative is the conformable derivative proposed by Khalil \cite{Khalil14}, which seeks to simplify aspects of fractional calculus while retaining features of classical differentiation. Unlike traditional definitions that often require stronger integrability conditions or special functions, the conformable derivative allows direct work with classically differentiable functions in a more accessible framework.

\subsection{Conformable  Derivative}
\begin{Definition}[Khalil~\cite{Khalil14}]
Let $\alpha \in (0,1]$. The Conformable  derivative of order $\alpha$ of a function $f: [0, \infty) \to \mathbb{R}$ is defined as:
\begin{equation}
T_\alpha(f)(t) := \lim_{\epsilon \to 0} \frac{f\left(t + \epsilon t^{1 - \alpha}\right) - f(t)}{\epsilon}, \quad t > 0.
\end{equation}
\end{Definition}
\begin{Remark}
If this limit exists, we say that $f$ is $\alpha$-differentiable.
\end{Remark}

This derivative can be viewed as a temporal deformation of the classical derivative, with parameter $\alpha$ controlling the local temporal change of scale. If the function $f$ is differentiable in the classical sense, it can also be expressed as:
\begin{equation}
T_\alpha(f)(t) = t^{1 - \alpha} \frac{df}{dt}(t), \label{conformable}
\end{equation}
clearly revealing its connection to the standard derivative. This property is particularly useful when seeking analytical solutions to fractional differential equations via integral transforms.

The conformable derivative preserves many classical properties of calculus, such as:

\begin{Theorem}
Let $\alpha \in (0,1)$, and assume $f$ and $g$ are $\alpha$-differentiable at a point $t$. Then the Conformable  derivative satisfies the following properties:
\begin{enumerate}
  \item \textbf{Linearity:} $T_\alpha(af + bg) = a T_\alpha(f) + b T_\alpha(g)$, for all $a, b \in \mathbb{R}$.
  \item \textbf{Power Rule:}  $T_\alpha(t^\lambda)=\lambda t^{\lambda-\alpha}$, for all $\lambda \in \mathbb{R}$.
  \item \textbf{Product Rule:} $T_\alpha(fg)(t) = f(t) T_\alpha(g)(t) + g(t) T_\alpha(f)(t)$.
  \item \textbf{Quotient Rule:} $T_\alpha\left(\frac{f}{g}\right)(t) = \frac{f(t) T_\alpha(g)(t) - g(t) T_\alpha(f)(t)}{g^2(t)}$, provided $g(t) \neq 0$.
\end{enumerate}
\end{Theorem}

In~\cite{Abdeljawad15}, T. Abdeljawad establishes the chain rule for the Conformable  derivative in the following Theorem:

\begin{Theorem}[~\cite{Abdeljawad15}]
Let $f, g:(0,\infty)\to\mathbb{R}$ be $\alpha$-differentiable functions with $\alpha \in (0,1)$. Then $f \circ g$ is $\alpha$-differentiable and for $t \ne 0$, $g(t) \ne 0$, we have:
\begin{equation}
T_\alpha(f \circ g)(t) = T_\alpha f(g(t)) \cdot T_\alpha g(t) \cdot g(t)^{\alpha-1}.
\end{equation}
\end{Theorem}

These theorems motivate their use in contexts that aim to extend classical differential analysis while preserving the operational structure.

\subsection{Conformable Exponential Function}

A central function in the theory of differential equations is the exponential. To ensure compatibility with the conformable derivative framework, the so-called \emph{conformable exponential function} is defined as the solution to:
\begin{equation}
T_\alpha y(t) = \lambda y(t), \quad y(0) = 1.
\end{equation}

The solution to this equation, based on the properties of $T_\alpha$, is:
\begin{equation}
y(t):= e_\alpha(\lambda, t) = e^{\frac{\lambda t^\alpha}{\alpha}},
\end{equation}
which reduces to the classical exponential $e^{\lambda t}$ when $\alpha = 1$. This function plays the same role as its classical counterpart in solving homogeneous linear equations and becomes a key component in the fundamental solutions of conformable differential systems.

\subsection{Conformable  Integral}

Following the natural development of the conformable derivative, it is reasonable to define an integral operator that acts as its inverse. This construction was proposed in analogy with the classical integration operator, respecting the new temporal framework introduced by the deformation $t^{1 - \alpha}$.

The conformable  integral of order $\alpha \in (0,1]$, also referred to as the conformable antiderivative, of a continuous function $f$ on the interval $[a,b]$, is defined as:
\begin{equation}
I^\alpha_a f(t) := \int_a^t \tau^{\alpha - 1} f(\tau)\, d\tau, \quad a \leq t \leq b,
\end{equation}
where the improper integral is understood in the Riemann sense and is assumed to exist.

This operator serves as the inverse of the operator $T_\alpha$ in the sense that, if $f$ is continuous and sufficiently regular, then:
\begin{equation}
T_\alpha \left( I^\alpha_a f \right)(t) = f(t), \quad I^\alpha_a  \left(T_\alpha f \right)(t) = f(t)-f(a).
\end{equation}

As discussed in recent works~\cite{Mohammadnezhad22}, this integral is especially useful in constructing explicit solutions to conformable differential equations via direct integration techniques.

An essential tool in the analysis of differential systems is integration by parts, which also admits an adapted version in the context of the conformable calculus. If $f, g \in C^1([a,b])$, then the following identity holds:
\begin{equation}
\int_a^b f(t) T_\alpha g(t)\, t^{\alpha - 1} dt = \left. f(t) g(t) \right|_a^b - \int_a^b g(t) T_\alpha f(t)\, t^{\alpha - 1} dt.
\end{equation}

This property is consistent with the symmetry of the derivative operator and the temporal weighting structure $t^{\alpha - 1}$ underlying the conformable formalism. It plays a key role in theoretical and variational analyses of fractional-order equations.

\subsection{Conformable Laplace Transform}

A fundamental tool for solving fractional differential equations is the Laplace transform. For the conformable derivative, a compatible version has been developed, known as the \emph{conformable Laplace transform}. This work adopts a more direct and operational approach, as suggested in~\cite{Mohammadnezhad22}.

\begin{Definition}[Conformable  Laplace Transform~\cite{Mohammadnezhad22}]
Let $0 < \alpha \le 1$ and $f: [0, \infty) \to \mathbb{R}$ be a real-valued function. Then the conformable Laplace transform of order $\alpha$ is defined as:
\begin{equation}\label{D.9}
\mathcal{L}_\alpha\{f(t)\}(s) := \int_0^\infty e^{-\frac{s t^\alpha}{\alpha}} f(t)\, d\alpha(t),
\end{equation}
where $d\alpha(t) = t^{\alpha-1} dt$.
\end{Definition}

This approach naturally enables analogs of classical properties, such as:
\begin{equation}\label{D.10}
\mathcal{L}_\alpha\{T_\alpha f(t)\}(s) = s \mathcal{L}_\alpha\{f(t)\}(s) - f(0),
\end{equation}
which facilitates the algebraic treatment of fractional differential equations in the frequency domain.
This definition allows the construction of a transform theory that mirrors the classical Laplace framework while adapting to the conformable context. As in the classical case, the conformable Laplace transform preserves important operational rules which are essential in solving differential equations. 

We now present some of the most useful properties of this transform, starting with linearity: 
Let $\mathcal{L}_\alpha$ denote the conformable Laplace transform. If $f(t)$ and $g(t)$ are functions for which $\mathcal{L}_\alpha\{f(t)\}(s)$ and $\mathcal{L}_\alpha\{g(t)\}(s)$ exist, and $a, b \in \mathbb{R}$, then:
\begin{equation}
\mathcal{L}_\alpha\{a f(t) + b g(t)\}(s) = a\, \mathcal{L}_\alpha\{f(t)\}(s) + b\, \mathcal{L}_\alpha\{g(t)\}(s).
\end{equation} 

In addition, the following theorem summarizes further key properties:

\begin{Theorem}
Suppose that $F_\alpha(s) = \mathcal{L}_\alpha[f(t)]$ exists for $\Re s > 0$.
\begin{enumerate}
    \item If $c$ is a constant, then 
    \begin{equation}
    \mathcal{L}_\alpha[c] = \frac{c}{s}.
    \end{equation}
    \item If $q$ is a constant, then
    \begin{equation}\label{eq13}
    \mathcal{L}_\alpha[t^q](s) = \alpha^{q/\alpha} \frac{\Gamma\left(1 + \frac{q}{\alpha}\right)}{s^{1 + q/\alpha}}.
    \end{equation}
\end{enumerate}
\end{Theorem}

\subsection{Linear Systems with Conformable  Derivatives and Delays}

In various applications—physics, biology, and cosmology—differential models include temporal delays $f(t-T)$ with $T>0$ \cite{hale1993,ruan2006,arino2006,retard}. The Laplace transform of $\theta(t-T)f(t-T)$ yields the factor $e^{-sT}$, commonly used to solve delay differential equations.

Interest in conformable derivatives with delays has grown because they model processes with limited memory.

\begin{Remark}[Remark 3.2 \cite{Mohammadnezhad22}]
Let $\mathcal{L}_\alpha \{y(t)\} = Y_\alpha(s)$ denote the conformable Laplace transform of order $\alpha$ applied to the function $y(t)$. When applying the transform to a delayed term $y(t - T)$, the result is:
\begin{equation} \label{eq:laplace-delay}
\mathcal{L}_\alpha \{ y(t - T) \}(s) = e^{-\frac{s T^\alpha}{\alpha}} Y_\alpha(s) + e^{-\frac{s T^\alpha}{\alpha}} \int_{-T}^0 e^{-s\frac{ t^\alpha}{\alpha}} y(t)\, d\alpha(t).
\end{equation}
\end{Remark}

This expression reveals that the transform of the delayed term consists of two components: a classical multiplicative factor $e^{-\frac{s T^\alpha}{\alpha}} Y_\alpha(s)$, analogous to the standard Laplace case, and an additional term that depends on the function’s values in the interval $t \in [-T, 0)$.
However, as in several previous studies on fractional systems with delay, it is common to assume that the function $y(t)$ satisfies:
\begin{equation}
y(t) = 0, \quad \text{for } t < 0,
\end{equation}
which causes the second integral in~\eqref{eq:laplace-delay} to vanish, thereby simplifying the transform to:
\begin{equation} \label{eq:laplace-delay-simplified}
\mathcal{L}_\alpha \{ y(t - T) \}(s) = e^{-\frac{s T^\alpha}{\alpha}} Y_\alpha(s).
\end{equation}

This assumption is not merely a mathematical convenience—it has practical and physical justification. For instance, in many models, the system is assumed to begin its evolution at $t = 0$, with no prior activity. This hypothesis yields cleaner transformation expressions and enables the analysis to focus on the system's future dynamics, thereby facilitating both theoretical insights and numerical implementation.

Moreover, this approach aligns with the modern perspective on conformable derivatives as operators with limited memory, in contrast to other fractional formulations (such as Caputo’s), in which the function's entire history influences its derivative.

The general form of a linear conformable fractional differential system with time delays is \cite{Mohammadnezhad22}
\begin{equation}\label{eq:sistema}
\begin{aligned}
D^{\alpha_1}_t x_1(t) &= a_{11}x_1(t)+\cdots+a_{1n}x_n(t)
  + b_{11}x_1(t-\tau_{11})+\cdots+b_{1n}x_n(t-\tau_{1n}),\\
D^{\alpha_2}_t x_2(t) &= a_{21}x_1(t)+\cdots+a_{2n}x_n(t)
  + b_{21}x_1(t-\tau_{21})+\cdots+b_{2n}x_n(t-\tau_{2n}),\\
&\ \vdots\\
D^{\alpha_n}_t x_n(t) &= a_{n1}x_1(t)+\cdots+a_{nn}x_n(t)
  + b_{n1}x_1(t-\tau_{n1})+\cdots+b_{nn}x_n(t-\tau_{nn}),
\end{aligned}
\end{equation}
with initial functions
\begin{equation}
x_i(t)=\phi_i(t),\qquad -\tau_{\max}\le t\le 0,\quad i=1,\dots,n,
\end{equation}
where $\tau_{\max}=\max_{i,j}\tau_{ij}$, $0<\alpha_i\le 1$, and
\begin{equation}
A=(a_{ij})_{i,j=1}^n,\qquad B=(b_{ij})_{i,j=1}^n\in\mathbb{R}^{n\times n}.
\end{equation}
The state vector is $x(t)=(x_1(t),\dots,x_n(t))^\top$ and each $\tau_{ij}>0$ denotes the delay affecting the coupling from $x_j$ to the $i$-th equation. 

\begin{Definition} The system described by equation~\eqref{eq:sistema} is classified as follows:
\begin{itemize}
    \item[(i)] The system is said to be \textbf{stable} if, for any initial condition $x_0 \in \mathbb{R}^n$, there exists a constant $\varepsilon > 0$ such that
    \begin{equation}
    \|x(t)\| \leq \varepsilon, \quad \forall t \geq 0.
    \end{equation}
    \item[(ii)] The system is said to be \textbf{asymptotically stable} if it is stable and additionally satisfies
    \begin{equation}
    \lim_{t \to \infty} \|x(t)\| = 0.
    \end{equation}
\end{itemize}
\end{Definition}
This formulation translates classical stability concepts into the framework of conformable derivatives, preserving the notion of bounded behavior over time and the asymptotic decay of solutions as $t \to \infty$.

Applying the fractional Laplace transform to system \eqref{eq:sistema} and using the fractional final-value theorem leads to the following stability criterion.

\begin{Theorem}[Theorem 3.3 \cite{Mohammadnezhad22}]
Let system \eqref{eq:sistema} be the linear conformable-delay system considered above. Define the characteristic matrix $M(s)\in\mathbb{C}^{n\times n}$ with entries
\begin{equation}
M_{ij}(s)=
\begin{cases}
s-a_{ii}-b_{ii}e^{-s\beta_{ii}\,\frac{\tau_{ii}^{\alpha_i}}{\alpha_i}}, & i=j,\\[4pt]
-\,a_{ij}-b_{ij}e^{-s\beta_{ij}\,\frac{\tau_{ij}^{\alpha_i}}{\alpha_i}}, & i\neq j,
\end{cases}
\end{equation}
where $\alpha_i\in(0,1]$, $a_{ij},b_{ij}\in\mathbb{R}$, $\beta_{ij}\in\mathbb{R}$, $s\in \mathbb{C}$, and $\tau>0$ is the delay. If every root $s$ of
\begin{equation}
\det\big(M(s)\big)=0
\end{equation}
has negative real part, then system \eqref{eq:sistema} is asymptotically stable.
\end{Theorem}

This formulation proves especially useful for solving fractional differential equations with delays using transform-based techniques, as shown in recent works~\cite{Mohammadnezhad22, retard}.

\section{Why the conformable Laplace transform is preferable to the classical Laplace transform?}

In this section, we explain why the \emph{classical} Laplace transform applied to conformable equations produces non‑algebraic transform relations, while the \emph{conformable} Laplace transform preserves the algebraic structure essential for practical solution.

\subsection{Model problem}
Consider the scalar conformable equation for a differentiable function $f(t)$: 
\begin{equation}\label{eq:scalar}
T_\alpha f(t)+a f(t)=b,\qquad f(0)=f_0,\qquad 0<\alpha\le1,
\end{equation}
with $T_\alpha f(t)$ given by
\begin{equation}\label{eq:conformable_def}
T_\alpha f(t)=t^{1-\alpha}f'(t).
\end{equation}
Equation \eqref{eq:scalar} is thus equivalent to
\begin{equation}\label{eq:scalar_explicit}
t^{1-\alpha}f'(t)+a f(t)=b.
\end{equation}

\subsection{Classical Laplace transform: the obstruction}
Applying $\mathcal{L}$ to \eqref{eq:scalar_explicit} gives
\begin{equation}\label{eq:classical1}
\mathcal{L}\{t^{1-\alpha}f'(t)\}+aF(s)=\frac{b}{s},\qquad F(s)=\mathcal{L}\{f(t)\},
\end{equation}
with the problematic term
\begin{equation}\label{eq:G}
G(s):=\mathcal{L}\{t^{1-\alpha}f'(t)\}.
\end{equation}

Two observations explain why $G(s)$ is non‑algebraic:
\begin{itemize}
  \item For integer $n$, multiplication by $t^n$ corresponds to differentiation in $s$:
  \begin{equation}\label{eq:integer_mult}
  \mathcal{L}\{t^n f(t)\}=(-1)^n\frac{d^n}{ds^n}F(s).
  \end{equation}
  For noninteger $\mu=1-\alpha$, however, we have a fractional $s$‑derivative
  \begin{equation}\label{eq:fractional_mult}
  \mathcal{L}\{t^\mu f(t)\}\sim D_s^\mu F(s).
  \end{equation}
  
  \item Multiplying \eqref{eq:scalar_explicit} by $t^{\alpha-1}$ yields
  \begin{equation}\label{eq:var_coeff}
  f'(t)+a\,t^{\alpha-1}f(t)=b\,t^{\alpha-1},
  \end{equation}
  whose transform is
  \begin{equation}\label{eq:classical2}
  sF(s)-f_0 + a\,\mathcal{L}\{t^{\alpha-1}f(t)\}=b\,\mathcal{L}\{t^{\alpha-1}\},
  \end{equation}
  again involving fractional $s$‑derivatives.
\end{itemize}

Thus, the classical transform introduces nonlocal operators in $s$, undermining the algebraic reduction central to Laplace methods.

\subsection{Conformable Laplace transform: resolution}

\begin{Proposition}
Let $0<\alpha\le 1$ and let $T_\alpha$ denote the conformable derivative.   
Consider the linear initial value problem
\begin{equation}\label{eq:conformable_ivp}
T_\alpha f(t)+a\,f(t)=b,\qquad f(0)=f_0,
\end{equation}
with constants $a,b,f_0\in\mathbb{R}$. Then the unique solution is
\begin{equation}\label{eq:conformable_solution}
f(t)=\frac{b}{a}+\Big(f_0-\frac{b}{a}\Big)\exp\!\Big(-\tfrac{a\,t^\alpha}{\alpha}\Big).
\end{equation}
\end{Proposition}

\begin{proof}
We provide a direct argument based on the conformable Laplace transform, highlighting its algebraic simplicity.

\noindent\textbf{Step 1. Transform the equation.}  
Applying $\mathcal{L}_\alpha$ to \eqref{eq:conformable_ivp}, we use the identity
\begin{equation}\label{comformable_Laplace}
\mathcal{L}_\alpha\{T_\alpha f(t)\}(s)=s\,\mathcal{L}_\alpha\{f(t)\}(s)-f(0).
\end{equation}
Thus, the transformed equation becomes
\begin{equation}
s\,\mathcal{L}_\alpha\{f\}(s)-f_0+a\,\mathcal{L}_\alpha\{f\}(s)=\frac{b}{s}.
\end{equation}

\noindent\textbf{Step 2. Solve algebraically in the transform domain.}  
Collecting terms gives
\begin{equation}
(s+a)\,\mathcal{L}_\alpha\{f\}(s)=f_0+\frac{b}{s}.
\end{equation}
Hence
\begin{equation}
\mathcal{L}_\alpha\{f\}(s)=\tfrac{f_0}{s+a}+\tfrac{b}{s(s+a)}.
\end{equation}
Using partial fractions,
\begin{equation}
\frac{b}{s(s+a)}=\tfrac{b}{a}\Big(\tfrac{1}{s}-\tfrac{1}{s+a}\Big),
\end{equation}
so that
\begin{equation}
\mathcal{L}_\alpha\{f\}(s)=\frac{b}{a}\cdot\frac{1}{s}+\Big(f_0-\frac{b}{a}\Big)\frac{1}{s+a}.
\end{equation}

\noindent\textbf{Step 3. Invert term by term.}  
The standard inverse pairs for $\mathcal{L}_\alpha$ are
\begin{equation}
\mathcal{L}_\alpha^{-1}\!\left\{\frac{1}{s}\right\}(t)=1,\qquad
\mathcal{L}_\alpha^{-1}\!\left\{\frac{1}{s+a}\right\}(t)=e^{-a t^\alpha/\alpha}.
\end{equation}
Applying these gives
\begin{equation}
f(t)=\tfrac{b}{a}+\Big(f_0-\tfrac{b}{a}\Big)e^{-a t^\alpha/\alpha},
\end{equation}
which matches \eqref{eq:conformable_solution}.

\noindent\textbf{Step 4. Uniqueness and stability.}  
Since each transform step is reversible under the usual hypotheses, the solution is unique. Moreover, for $a>0$ the exponential term decays as $t\to\infty$, so
\begin{equation}
\lim_{t\to\infty} f(t)=\frac{b}{a},
\end{equation}
showing convergence to a stable equilibrium under constant forcing.
\end{proof}

\begin{Remark}
The conformable Laplace transform is particularly effective because it maps the conformable derivative to the algebraic form \eqref{eq:scalar_explicit}, preserving the classical workflow of transform–solve–invert. By contrast, applying the classical Laplace transform directly to the equivalent form \eqref{eq:scalar_explicit} introduces noninteger powers of $t$, leading to fractional derivatives in the transform variable. This complicates the algebra and obscures the structure of the solution, whereas the conformable transform retains a purely algebraic representation.
\end{Remark}

\subsection{Solution of a Conformable  Differential Equation Using the Conformable Laplace Transform}

We solve the Conformable  PDE
\begin{equation}\label{eq:PDE}
T_{\alpha} u(x,t)=D_x^2 u(x,t)-2u(x,t),\qquad 0<\alpha\le1,
\end{equation}
with initial condition
\begin{equation}\label{eq:IC}
u(x,0)=e^x,
\end{equation}
where $D_x^2$ is the spatial derivative of second order and $T_{\alpha}$ is now promoted to the partial in time conformable derivative.
The conformable Laplace transform $\mathcal{L}_\alpha$ is chosen so that the conformable derivative is transformed to the algebraic form \eqref{eq:scalar_explicit}. 

\textbf{Step 1: Transform in $t$.} Apply $\mathcal{L}_\alpha$ to \eqref{eq:PDE} (treating $x$ as a parameter):
\begin{equation}
s\,\mathcal{L}_\alpha\{u(x,t)\}-u(x,0)=\mathcal{L}_\alpha\{D_x^2 u(x,t)\}-2\mathcal{L}_\alpha\{u(x,t)\}.
\end{equation}
With $u(x,0)=e^x$ this yields
\begin{equation}\label{eq:LaplaceSol}
\mathcal{L}_\alpha\{u(x,t)\}=\frac{e^x}{s+2}+\frac{1}{s+2}\,\mathcal{L}_\alpha\{D_x^2 u(x,t)\}.
\end{equation}

\textbf{Step 2: Inverse transform and ansatz.} Using the conformable pair
\begin{equation}
\mathcal{L}_\alpha\{e^{-a t^\alpha/\alpha}\}=\frac{1}{s+a},
\end{equation}
we have
\begin{equation}
\mathcal{L}_\alpha^{-1}\!\left\{\frac{e^x}{s+2}\right\}=e^x e^{-2t^\alpha/\alpha}.
\end{equation}
So \eqref{eq:LaplaceSol} becomes
\begin{equation}
u(x,t)=e^x e^{-2t^\alpha/\alpha}+\mathcal{L}_\alpha^{-1}\!\left\{\frac{1}{s+2}\,\mathcal{L}_\alpha\{D_x^2 u\}\right\}.
\end{equation}
Motivated by the initial data, set $u(x,t)=e^x v(t)$. Since $D_x^2 e^x=e^x$, substitution into \eqref{eq:PDE} gives the scalar conformable ODE
\begin{equation}\label{eq:scalarODE}
T_{\alpha} v(t)=-v(t),\qquad v(0)=1.
\end{equation}

\textbf{Step 3: Solve the scalar ODE.} Using $T_{\alpha} v(t)=t^{1-\alpha}v'(t)$,
\begin{equation}
t^{1-\alpha}v'(t)=-v(t)\quad\Longrightarrow\quad \frac{v'(t)}{v(t)}=-t^{\alpha-1}.
\end{equation}
Integrating from $0$ to $t$ yields
\begin{equation}
\ln v(t)=-\frac{t^\alpha}{\alpha},\qquad v(t)=\exp\!\Big(-\frac{t^\alpha}{\alpha}\Big).
\end{equation}
Hence
\begin{equation}\label{eq:final}
u(x,t)=e^x\exp\!\Big(-\frac{t^\alpha}{\alpha}\Big).
\end{equation}

\textbf{Step 4: Series construction and induction.} Expand the factor $e^{t^\alpha/\alpha}$ to obtain the series representation
\begin{equation}
u(x,t)=\sum_{k=0}^\infty u_k(x,t),
\end{equation}
where
\begin{equation}\label{eq:uk_def}
u_k(x,t):=e^x e^{-2t^\alpha/\alpha}\,\frac{t^{k\alpha}}{\alpha^k k!},\qquad k\ge0.
\end{equation}
A direct computation gives the recurrence
\begin{equation}
u_{k+1}(x,t)=\frac{t^\alpha}{\alpha(k+1)}\,u_k(x,t),
\end{equation}
and induction on $k$ verifies \eqref{eq:uk_def} for all $k$. Summing the series termwise yields \eqref{eq:final}.

\begin{Proposition}
Let $0<\alpha\le1$. The initial value problem \eqref{eq:PDE}--\eqref{eq:IC} has the unique solution
\begin{equation}
u(x,t)=e^x\exp\!\Big(-\dfrac{t^\alpha}{\alpha}\Big).
\end{equation}
Equivalently, the solution separates as $u(x,t)=e^x v(t)$ with $v(t)=\exp(-t^\alpha/\alpha)$. The series $\sum_{k\ge0}u_k(x,t)$ defined in \eqref{eq:uk_def} converges (formally) to this closed form and satisfies the recurrence
\begin{equation}
u_{k+1}=\tfrac{t^\alpha}{\alpha(k+1)}\,u_k,\qquad u_0=e^x e^{-2t^\alpha/\alpha}.
\end{equation}
\end{Proposition}

\begin{proof}
Apply $\mathcal{L}_\alpha$ to \eqref{eq:PDE} and use $\mathcal{L}_\alpha\{T_{\alpha} u\}=s\mathcal{L}_\alpha\{u\}-u(x,0)$ to obtain \eqref{eq:LaplaceSol}. The separation ansatz reduces the PDE to the scalar ODE \eqref{eq:scalarODE}, whose solution is $v(t)=\exp(-t^\alpha/\alpha)$ by direct integration. Substituting back gives \eqref{eq:final}. The series representation follows by expanding $e^{t^\alpha/\alpha}$ and verifying the recurrence and term formula \eqref{eq:uk_def} by induction; summation reproduces the closed form. Uniqueness follows from linearity and the invertibility of $\mathcal{L}_\alpha$ on the transformable class.
\end{proof}

\textbf{Interpretation and limiting cases}
\begin{itemize}
  \item \textbf{Spatial profile.} The solution preserves the spatial shape $e^x$ for all $t$.
  \item \textbf{Time behavior.} The amplitude decays as a fractional exponential $e^{-t^\alpha/\alpha}$.
  \item \textbf{Classical limit.} For $\alpha=1$ the decay reduces to $e^{-t}$, recovering the classical result.
  \item \textbf{Short‑time expansion.} For $0<\alpha<1$ the decay is slower near $t=0$; the Taylor expansion
  \begin{equation}
  e^{-t^\alpha/\alpha}=1-\frac{t^\alpha}{\alpha}+\frac{t^{2\alpha}}{2\alpha^2}+\cdots
  \end{equation}
  matches the iterative terms $u_0,u_1,\dots$ obtained from the series construction.
\end{itemize}

\section{Differential Equation with First Order Derivative and Time Delay}

Consider the nonhomogeneous delay differential equation \cite{retard}:
\begin{equation}\label{eq:first_order_delay}
y'(t) + a\, y(t - T) = b, \qquad y(0)=y_0,\quad y(t)=0 \;\text{for } t<0,
\end{equation}
where $T>0$ and $a, b\in\mathbb{R}$ are constants. Applying the Laplace transform to \eqref{eq:first_order_delay} gives
\begin{equation}\label{eq:laplace_first}
s Y(s) - y(0) + a\, Y(s) e^{-sT} = \frac{b}{s},
\end{equation}
where $Y(s)=\mathcal{L}\{y(t)\}$. Solving for $Y(s)$ yields
\begin{equation}\label{eq:Y_first}
Y(s) = \frac{y(0)}{s\left(1+\tfrac{a e^{-sT}}{s}\right)} + \tfrac{b}{s^2\left(1+\tfrac{a e^{-sT}}{s}\right)}.
\end{equation}

\begin{Remark}\label{rem:convergence_first}
Let $c>0$. Then $0<\left|\tfrac{a e^{-sT}}{c}\right|<1$ holds for all $c\ge |a|>0$, since $|a/c|\le 1$ and $\left|e^{-sT}\right|<1$ for $\Re s>0$.
\end{Remark}

Expanding the denominator in \eqref{eq:Y_first} as a geometric series under condition $0<\left|\tfrac{a e^{-sT}}{s}\right|<1$, we obtain
\begin{equation}\label{eq:Y_series_first}
Y(s) = y(0)\sum_{j=0}^\infty (-a)^j \frac{e^{-sjT}}{s^{j+1}}
+ b\sum_{j=0}^\infty (-a)^j \frac{e^{-sjT}}{s^{j+2}}.
\end{equation}

Using the inversion formula
\begin{equation}
\mathcal{L}^{-1}\!\left[\frac{e^{-sjT}}{s^{m+1}}\right](t)
= \frac{(t-jT)^m}{\Gamma(m+1)}\,\theta(t-jT),
\end{equation}
we recover the explicit solution
\begin{equation}\label{eq:series_first}
y(t) = y(0)\sum_{j=0}^\infty (-a)^j  \frac{(t-jT)^j}{\Gamma(j+1)}\,\theta(t-jT)
+ b\sum_{j=0}^\infty (-a)^j  \frac{(t-jT)^{j+1}}{\Gamma(j+2)}\,\theta(t-jT).
\end{equation}

Since $\theta(t-jT)=0$ whenever $j>t/T$, the sums truncate at $j=\lfloor t/T\rfloor$, giving the finite representation
\begin{equation}\label{eq:finite_first}
y(t)=\sum_{j=0}^{\lfloor t/T\rfloor}\left(y(0)+\tfrac{b\,(t-jT)}{j+1}\right)(-a)^j  \frac{(t-jT)^j}{\Gamma(j+1)}\,\theta(t-jT).
\end{equation}
In particular, on $[0,T)$ only $j=0$ contributes, so
\begin{equation}\label{eq:first_interval}
y(t)=y(0)+b\,t,\qquad t\in[0,T).
\end{equation}

\begin{Theorem}[Series solution for the first-order delay equation]
Let $a, T>0$, $b\in\mathbb{R}$, and suppose $y(t)$ is causal, locally absolutely continuous, and of exponential order so that the classical Laplace transform $\mathcal{L}\{y\}$ exists for $\Re s$ sufficiently large. Consider the delay differential equation
\begin{equation}\label{eq:first_order_delay_repeat}
y'(t) + a  y(t - T) = b, \qquad y(0)=y_0,\quad y(t)=0 \;\text{for } t<0.
\end{equation}
If the convergence condition
\begin{equation}\label{eq:conv_first}
\left|\tfrac{a e^{-sT}}{s}\right|<1
\end{equation}
holds on the vertical line $\Re s=\sigma$ used for inversion, then the solution is given by the finite series
\begin{equation}\label{eq:finite_first_repeat}
y(t)=\sum_{j=0}^{\lfloor t/T\rfloor}\left(y(0)+\tfrac{b\,(t-jT)}{j+1}\right)(-a)^j  \frac{(t-jT)^j}{\Gamma(j+1)}\,\theta(t-jT).
\end{equation}
In particular, on the first interval $[0,T)$ the solution reduces to
\begin{equation}\label{eq:first_interval_repeat}
y(t)=y(0)+b\,t,\qquad t\in[0,T).
\end{equation}
\end{Theorem}

\begin{proof}
Starting from the transformed equation \eqref{eq:Y_first}, expansion under condition \eqref{eq:conv_first} yields the series \eqref{eq:Y_series_first}. Each term of the form $s^{-(j+1)}e^{-sjT}$ corresponds, via the inversion formula
\begin{equation}
\mathcal{L}^{-1}\!\left[\frac{e^{-sjT}}{s^{j+1}}\right](t)
= \frac{(t-jT)^j}{\Gamma(j+1)}\,\theta(t-jT),
\end{equation}
to a delayed monomial weighted by $y(0)(-a)^j$. 

Similarly, each term of the form $s^{-(j+2)}e^{-sjT}$ corresponds to
\begin{equation}
\frac{(t-jT)^{j+1}}{\Gamma(j+2)}\,\theta(t-jT),
\end{equation}
multiplied by $b(-a)^j$. Collecting contributions yields the series representation \eqref{eq:series_first}.

Since $\theta(t-jT)=0$ whenever $j>t/T$, only finitely many terms contribute for each fixed $t>0$, truncating the series at $j=\lfloor t/T\rfloor$ and giving \eqref{eq:finite_first_repeat}. On the first interval $[0,T)$, only $j=0$ contributes, which reduces the solution to \eqref{eq:first_interval_repeat}.

Thus, the series representation \eqref{eq:finite_first_repeat} satisfies the delay differential equation \eqref{eq:first_order_delay_repeat} together with the initial conditions, completing the proof.
\end{proof}

\section{Differential Equation with Conformable  Derivative and Time Delay}

We consider the initial‑value problem with a conformable derivative and a fixed delay:
\begin{equation}\label{eq:model} 
T_\alpha y(t) + a\,y(t-T)=b,\qquad y(0)=y_0,\qquad y(t)=0\ \text{for }t<0,
\end{equation}
where $0<\alpha\le1$, $T>0$, and $a,b,y_0\in\mathbb{R}$.

\textbf{Step 1: Transform the equation.}
Apply $\mathcal{L}_\alpha$ to both sides of \eqref{eq:model}. Using the delay (shift) property for the conformable transform (valid under the causal assumption $y(t)=0$ for $t<0$),
\begin{equation}
\mathcal{L}_\alpha\{y(t-T)\}(s)=e^{-\frac{sT^\alpha}{\alpha}}\,\mathcal{L}_\alpha\{y\}(s),
\end{equation}
we obtain the algebraic equation
\begin{equation}\label{eq:transform_domain}
s\,Y_\alpha(s)-y_0 + a\,e^{-\frac{sT^\alpha}{\alpha}}\,Y_\alpha(s)=\frac{b}{s},
\end{equation}
where $Y_\alpha(s):=\mathcal{L}_\alpha\{y(t)\}(s)$.

\textbf{Step 2: Solve algebraically in the transform domain.}
Collect terms in $Y_\alpha(s)$ and solve:
\begin{equation}
\left(s + a\,e^{-\frac{sT^\alpha}{\alpha}}\right)Y_\alpha(s)=y_0+\frac{b}{s}.
\end{equation}
Equivalently,
\begin{equation}\label{eq:Yalpha_exact}
Y_\alpha(s)=\tfrac{y_0}{s\left(1+\dfrac{a\,e^{-\frac{sT^\alpha}{\alpha}}}{s}\right)}+\tfrac{b}{s^2\left(1+\dfrac{a\,e^{-\frac{sT^\alpha}{\alpha}}}{s}\right)}.
\end{equation}

\textbf{Step 3: Geometric expansion (region of validity).}
For values of $s$ with
\begin{equation}
\left|\frac{a\,e^{-\frac{sT^\alpha}{\alpha}}}{s}\right|<1
\end{equation}
(e.g. sufficiently large $\Re s$), expand the factor $1/(1+z)$ as a convergent geometric series with $z=\dfrac{a\,e^{-\frac{sT^\alpha}{\alpha}}}{s}$. Substituting into \eqref{eq:Yalpha_exact} yields
\begin{equation}\label{eq:Yalpha_series}
Y_\alpha(s)
= y_0\sum_{j=0}^\infty (-a)^j  \frac{e^{-\frac{j s T^\alpha}{\alpha}}}{s^{\,j+1}}
+ b\sum_{j=0}^\infty (-a)^j  \frac{e^{-\frac{j s T^\alpha}{\alpha}}}{s^{\,j+2}}.
\end{equation}

\textbf{Step 4: Inverse transform term by term.} Substituting $q= \alpha p$ in \eqref{eq13}, we obtain the standard conformable inverse pair
\begin{equation}
\mathcal{L}_\alpha\!\left\{\frac{t^{\alpha p}}{\alpha^p\Gamma(1+p)}\right\}(s)=s^{-(1+p)},\qquad p>-1.
\end{equation}
We invert each term in \eqref{eq:Yalpha_series} 
together with the delay‑shift identity
\begin{equation}\label{action-of-L}
\mathcal{L}_\alpha\!\left\{\frac{(t-j T)^{\alpha p}}{\alpha^p\Gamma(1+p)}\,\theta(t-jT)\right\}(s)
= s^{-(1+p)}\,e^{-\frac{j s T^\alpha}{\alpha}}.
\end{equation}
Applying $\mathcal{L}_\alpha^{-1}$ to each series term gives, for integer $j\ge0$,
\begin{equation}
\mathcal{L}_\alpha^{-1}\!\left\{\frac{e^{-\frac{j s T^\alpha}{\alpha}}}{s^{\,j+1}}\right\}
=\frac{(t-jT)^{\alpha j}}{\alpha^j \Gamma(j+1)}\,\theta(t-jT),
\end{equation}
and
\begin{equation}
\mathcal{L}_\alpha^{-1}\!\left\{\frac{e^{-\frac{j s T^\alpha}{\alpha}}}{s^{\,j+2}}\right\}
=\frac{(t-jT)^{\alpha(j+1)}}{\alpha^{j+1}\Gamma(j+2)}\,\theta(t-jT).
\end{equation}

\textbf{Step 5: Finite sum and final time‑domain expression.}
Because $\theta(t-jT)=0$ whenever $j>\lfloor t/T\rfloor$, the infinite series truncates for each fixed $t>0$. Combining the inverted terms yields the explicit time‑domain representation
\begin{equation}\label{eq:prop_series}
y(t)=\sum_{j=0}^{\lfloor t/T\rfloor}(-a)^j \left[
y_0\,\frac{(t-jT)^{\alpha j}}{\alpha^j \Gamma(j+1)}+b\,\frac{(t-jT)^{\alpha(j+1)}}{\alpha^{j+1}\Gamma(j+2)}
\right]\theta(t-jT).
\end{equation}
In particular, for $t\in[0,T)$ only the $j=0$ term contributes and
\begin{equation}
y(t)=y_0+\tfrac{b}{\alpha}\,t,\qquad t\in[0,T) \label{eq:prop_k=0}.
\end{equation}

\begin{Proposition}[Series solution for the conformable delay equation]
Let $0<\alpha\le1$, $T>0$, and let $y(t)$ satisfy \eqref{eq:model}  with causal initial data $y(t)=0$ for $t<0$. Assume the conformable Laplace transform of $y$ exists and that the geometric expansion in \eqref{eq:Yalpha_series} is valid for sufficiently large $\Re s$. Then for every $t>0$, $y(t)$ is given by the finite sum \eqref{eq:prop_series}. Moreover, on the first interval $t\in[0,T)$ the solution reduces to \eqref{eq:prop_k=0}.
\end{Proposition}

\begin{proof}
Starting from \eqref{eq:transform_domain}, we solved algebraically for $Y_\alpha(s)$ and expanded the rational factor as a geometric series under the stated smallness condition on $a e^{-sT^\alpha/\alpha}/s$. Each term in the resulting series is of the form $s^{-(j+1)}e^{-j s T^\alpha/\alpha}$ or $s^{-(j+2)}e^{-j s T^\alpha/\alpha}$, whose inverse conformable transforms are the delayed fractional monomials displayed above. The Heaviside factors truncate the sum at $j=\lfloor t/T\rfloor$, producing the finite sum \eqref{eq:prop_series}. The reduction on $[0,T)$ follows by setting $\lfloor t/T\rfloor=0$.
\end{proof}

\begin{Remark}

\begin{enumerate}
  \item \textbf{Region of validity.} The geometric expansion used in \eqref{eq:Yalpha_series} requires $\left|a\,e^{-\frac{sT^\alpha}{\alpha}}/s\right|<1$. This condition holds for sufficiently large $\Re s$ and justifies termwise inversion by analytic continuation or by Bromwich‑type contour arguments in the conformable transform setting.
  \item \textbf{Causality and finiteness.} Causality (the assumption $y(t)=0$ for $t<0$) ensures the delay‑shift identity and guarantees that, for each fixed $t$, only finitely many delayed monomials are active; hence, the time‑domain representation is a finite sum.
  \item \textbf{Comparison with classical Laplace.} Attempting the same procedure with the classical Laplace transform applied to the equivalent form $t^{1-\alpha}y'(t)+a y(t-T)=b$ produces factors $t^{\beta}$ with noninteger $\beta$. Their classical transforms involve fractional $s$-derivatives rather than elementary algebraic operations on $Y(s)$, which complicates inversion. The conformable transform restores the algebraic transform–solve–invert workflow and yields the explicit series above.
\end{enumerate}
\end{Remark}

\section{Analytic $b(t)$}

In this section, we further extend the scope by analyzing the analytic $b(t)$. 

\subsection{Conformable derivatives}
 
Now, we extend the model \eqref{eq:model} by considering a time depending right hand side $b(t)$
 \begin{equation}\label{eq:model_nl}
T_\alpha y(t)+a\,y(t-T)=b(t), \qquad y(0)=y_0,\qquad y(t)=0\ \text{for }t<0,
\end{equation}
where $0<\alpha\le1$, $T>0$, and $a,b,y_0\in\mathbb{R}$.

Let us assume that $b(t)$ is analytical and that have the expansion 
\begin{equation}
    b(t)= \sum_{k=0}^{\infty} b_k \left(\frac{t^{\alpha}}{\alpha}\right)^k. \label{nl}
\end{equation}

\subsubsection{Laplace-domain representation}

Applying the Laplace transform $\mathcal{L}_\alpha$ we obtain 
\begin{equation}\label{eq:Yalpha_exact_nl}
Y_\alpha(s)=\tfrac{y_0}{s\left(1+\dfrac{a\,e^{-\frac{sT^\alpha}{\alpha}}}{s}\right)}+\tfrac{\sum_{k=0}^{\infty} b_k\Gamma(k+1) s^{-(2+ k) }}{\left(1+\dfrac{a\,e^{-\frac{sT^\alpha}{\alpha}}}{s}\right)}.
\end{equation}
For values of $s$ with
\begin{equation}
\left|\frac{a\,e^{-\frac{sT^\alpha}{\alpha}}}{s}\right|<1
\end{equation}
(e.g. sufficiently large $\Re s$), expand the factor $1/(1+z)$ as a convergent geometric series with $z=\dfrac{a\,e^{-\frac{sT^\alpha}{\alpha}}}{s}$. Substituting into \eqref{eq:Yalpha_exact_nl} we obtain 
\begin{equation}\label{eq:Yalpha_series_nl}
Y_\alpha(s)
= y_0\sum_{j=0}^\infty (-a)^j  \frac{e^{-\frac{j s T^\alpha}{\alpha}}}{s^{\,j+1}}
+ \sum_{j=0}^{\infty}\sum_{k=0}^{\infty} (-a)^{j} e^{-\frac{j s T^\alpha}{\alpha}} b_k \Gamma(k+1)  s^{-(j+k+2)}.
\end{equation}
Using the property \eqref{action-of-L}, we obtain
\begin{equation}\label{eq:prop_series_nl}
y(t)=\sum_{j=0}^{\lfloor t/T\rfloor}(-a)^j 
\left[y_0\,\frac{(t-jT)^{\alpha j}}{\alpha^j \Gamma(j+1)}  + \sum_{k=0}^{\infty}
b_k\,\frac{\Gamma(k+1) (t-jT)^{\alpha(j+k+1)}}{\alpha^{j+k+1}\Gamma(j+k+2)}\right] \theta(t-jT) .
\end{equation}

The external sum is finite by the Heaviside functions. 
Examples of functions that admit the expansion \eqref{nl} are:

\begin{itemize}
    \item Polynomial functions, e.g.
    \begin{equation}
    b(t)=c_0+c_1\frac{t^\alpha}{\alpha}+c_2\left(\frac{t^\alpha}{\alpha}\right)^2+\cdots,
    \end{equation}
    which correspond to finite truncations of the series.
    
    \item Exponential functions in conformable form,
    \begin{equation}
    b(t)=e^{\lambda \frac{t^\alpha}{\alpha}}
    =\sum_{k=0}^\infty \frac{\lambda^k}{k!}\left(\frac{t^\alpha}{\alpha}\right)^k,
    \end{equation}
    defining the conformable exponential $e_\alpha(\lambda, t)$.

    \item Trigonometric functions in conformable notation:
    \begin{equation}
    \sin_\alpha(t):=\sin\!\left(\frac{t^\alpha}{\alpha}\right)
    =\sum_{k=0}^\infty (-1)^k \frac{\left(\tfrac{t^\alpha}{\alpha}\right)^{2k+1}}{(2k+1)!},
    \end{equation}
    \begin{equation}
    \cos_\alpha(t):=\cos\!\left(\frac{t^\alpha}{\alpha}\right)
    =\sum_{k=0}^\infty (-1)^k \frac{\left(\tfrac{t^\alpha}{\alpha}\right)^{2k}}{(2k)!}.
    \end{equation}

    \item More generally, any entire function $f(z)$ with Taylor expansion
    \begin{equation}
    f(z)=\sum_{k=0}^\infty c_k z^k
    \end{equation}
    admits the representation $b(t)=f\!\left(\tfrac{t^\alpha}{\alpha}\right)$.
\end{itemize}

Thus, the framework encompasses a wide class of analytic forcings, from simple polynomials to transcendental functions, all of which are reducible to the series form \eqref{nl}. In particular, the conformable sine and cosine functions $\sin_\alpha(t)$ and $\cos_\alpha(t)$ provide natural oscillatory forcings within this setting.

\begin{Theorem}[Series solution for the conformable delay equation]
Let $0<\alpha\le1$, $a>0$, $T>0$, and suppose $y(t)$ is causal, locally absolutely continuous, and of exponential order so that $\mathcal{L}_\alpha\{y\}$ exists for $\Re s$ sufficiently large. Assume $b(t)$ is analytic with expansion \eqref{nl}. 
If the convergence condition
\begin{equation}\label{eq:conv_conf}
\left|\tfrac{a\,e^{-sT^\alpha/\alpha}}{s}\right|<1
\end{equation}
holds on the vertical line $\Re s=\sigma$ used for inversion, then the series
\begin{equation}\label{eq:series_conf}
y(t)=\sum_{j=0}^{\lfloor t/T\rfloor}(-a)^j 
\left[y_0\,\frac{(t-jT)^{\alpha j}}{\alpha^j \Gamma(j+1)}  
+ \sum_{k=0}^{\infty} b_k\,\frac{\Gamma(k+1)\,(t-jT)^{\alpha(j+k+1)}}{\alpha^{j+k+1}\Gamma(j+k+2)}\right] \theta(t-jT)
\end{equation}
is the exact solution of the conformable fractional-delay equation \eqref{eq:model_nl}.
\end{Theorem}

\begin{proof}
Starting from the transformed equation \eqref{eq:Yalpha_exact_nl}, the denominator is expanded as a geometric series under the condition \eqref{eq:conv_conf}. This produces the double series \eqref{eq:Yalpha_series_nl}, which converges uniformly on the inversion line. Uniform convergence ensures that termwise inversion is valid.

Each term of the form $s^{-(j+1)}e^{-\frac{j s T^\alpha}{\alpha}}$ in the first sum of \eqref{eq:Yalpha_series_nl} corresponds, via the inversion property of $\mathcal{L}_\alpha$, to the delayed conformable monomial
\begin{equation}
\frac{(t-jT)^{\alpha j}}{\alpha^j \Gamma(j+1)}\,\theta(t-jT).
\end{equation}
Similarly, each term of the form $s^{-(j+k+2)}e^{-\frac{j s T^\alpha}{\alpha}}$ in the second sum corresponds to
\begin{equation}
\frac{(t-jT)^{\alpha(j+k+1)}}{\alpha^{j+k+1}\Gamma(j+k+2)}\,\theta(t-jT),
\end{equation}
multiplied by $b_k\Gamma(k+1)$. Collecting contributions yields precisely the series \eqref{eq:series_conf}.

To verify that this series solves \eqref{eq:model_nl}, recall that the conformable derivative acts on fractional monomials as
\begin{equation}
T_\alpha\!\left(\frac{(t-jT)^{\alpha m}}{\alpha^m \Gamma(m+1)}\right)
= \frac{(t-jT)^{\alpha(m-1)}}{\alpha^{m-1}\Gamma(m)}.
\end{equation}
On the first interval $[0, T)$, only the $j=0$ terms contribute, and differentiation reconstructs the forcing expansion $b(t)$ while the delay term vanishes. On a general interval $[jT,(j+1)T)$, the derivative of the initial-condition term cancels with the delay contribution from the previous epoch, while the forcing terms reproduce $b(t)$ exactly. The prefactor $(-a)^j$ guarantees recursive cancellation across successive intervals.

By induction over $j$, the solution holds for all $t>0$. Finally, the initial conditions are satisfied: at $t=0$ only $y(0)$ remains, and for $t<0$ all terms vanish due to the Heaviside factors. Hence, the series \eqref{eq:series_conf} is the exact solution of the conformable fractional-delay equation.
\end{proof}

\textbf{Conformable exponential.}
For $e^z=\sum_{k=0}^\infty \frac{z^k}{k!}$ we have $b_k=\tfrac{1}{k!}$, hence
\begin{equation}
y(t)=\sum_{j=0}^{\lfloor t/T\rfloor}(-a)^j\Bigg[
y_0\,\frac{(t-jT)^{\alpha j}}{\alpha^j j!}
+\sum_{k=0}^\infty \frac{(t-jT)^{\alpha(j+k+1)}}{\alpha^{j+k+1}(k+1)!}
\Bigg]\theta(t-jT).
\end{equation}

\textbf{Conformable sine.}
For $\sin z=\sum_{k=0}^\infty (-1)^k \frac{z^{2k+1}}{(2k+1)!}$ we have $b_{2k+1}=\tfrac{(-1)^k}{(2k+1)!}$, $b_{2k}=0$, hence
\begin{equation}
y(t)=\sum_{j=0}^{\lfloor t/T\rfloor}(-a)^j\Bigg[
y_0\,\frac{(t-jT)^{\alpha j}}{\alpha^j j!}
+\sum_{k=0}^\infty (-1)^k\,\frac{(t-jT)^{\alpha(j+2k+2)}}{\alpha^{j+2k+2}(j+2k+2)}
\Bigg]\theta(t-jT).
\end{equation}

\textbf{Conformable cosine.}
For $\cos z=\sum_{k=0}^\infty (-1)^k \frac{z^{2k}}{(2k)!}$ we have $b_{2k}=\tfrac{(-1)^k}{(2k)!}$, $b_{2k+1}=0$, hence
\begin{equation}
y(t)=\sum_{j=0}^{\lfloor t/T\rfloor}(-a)^j\Bigg[
y_0\,\frac{(t-jT)^{\alpha j}}{\alpha^j j!}
+\sum_{k=0}^\infty (-1)^k\,\frac{(t-jT)^{\alpha(j+2k+1)}}{\alpha^{j+2k+1}(j+2k+1)}
\Bigg]\theta(t-jT).
\end{equation}

\subsubsection{Analytic caveats and region of validity}

\begin{itemize}
  \item \textbf{Geometric expansion and domain of convergence.}  
    The expansion
    \begin{equation}
    \frac{1}{1+z}=\sum_{j=0}^\infty (-1)^j z^j,\qquad |z|<1,
    \end{equation}
    was applied in the conformable $s$-domain with
    $z=\dfrac{a\,e^{-sT^\alpha/\alpha}}{s}$. Thus the series representation of $Y_\alpha(s)$ is valid for those $s$ satisfying
    \begin{equation}
    \left|\tfrac{a\,e^{-sT^\alpha/\alpha}}{s}\right|<1.
    \end{equation}
    In practice, this inequality holds for sufficiently large $\Re s$. One then obtains a representation of $Y_\alpha(s)$ that is analytic in a right half‑plane; inversion is justified by analytic continuation or by deforming the inversion contour into that half‑plane.

  \item \textbf{Justification of termwise inversion.}  
    To invert the geometric series termwise, we require uniform convergence of the series on a vertical strip $\Re s=\sigma$ used for inversion (or domination by an $L^1$-integrable majorant on the Bromwich contour). Under the hypotheses that $y(t)$ is of exponential order and that the geometric factor satisfies the above smallness condition for $\Re s\ge\sigma$, the series converges uniformly on the contour, and Fubini / dominated convergence justifies termwise inversion. Equivalently, one may invert each rational–exponential term and then sum the resulting delayed monomials; for fixed $t$, the Heaviside factors truncate the sum, producing a finite, exact expression.

  \item \textbf{Bromwich contour and conformable inversion.}  
    The conformable inverse transform is implemented by a Bromwich‑type integral in the conformable $s$-plane. When the transform is analytic in a right half‑plane containing the vertical inversion line, contour deformation and residue calculus (or standard inverse‑transform arguments) validate the passage from the $s$-domain series to the time‑domain series.

  \item \textbf{Branch cuts and multi‑valuedness in the classical approach.}  
    The classical Laplace transform of noninteger powers produces factors $s^{-\gamma}$ with $\gamma\not\in\mathbb{Z}$. Such factors are multi-valued and require branch cuts; the inversion integral must be defined on a chosen branch, and it typically involves Hankel or Mellin contours. This introduces additional analytic complexity (branch selection, cut placement, and associated contributions) and often yields solutions expressed in terms of special functions rather than elementary kernels.

  \item \textbf{Regularity and growth assumptions.}  
    The formal operator identities and inversion steps above assume:
    \begin{enumerate}
      \item $y(t)$ is locally absolutely continuous on $[0,\infty)$ and of exponential order, so that $\mathcal{L}\{y\}$ and $\mathcal{L}_\alpha\{y\}$ exist for $\Re s$ large enough;
      \item the geometric factor $a e^{-sT^\alpha/\alpha}/s$ is uniformly small on the chosen inversion contour;
      \item termwise differentiation and integration under the integral sign are justified by dominated convergence.
    \end{enumerate}
    These hypotheses are standard in transform methods and should be verified in concrete applications.
\end{itemize}
When the above analytic conditions hold, the conformable transform route yields explicit, finite (for fixed $t$) series solutions obtained by straightforward inversion of elementary terms. If the conditions fail, one must resort to analytic continuation, asymptotic expansions, or numerical inversion techniques.

\subsection{Caputo derivatives}
For comparison, consider the Caputo fractional-delay problem
\begin{equation}\label{eq:Caputo_model_nl}
{}^{C}D_t^{\alpha} y(t)+a\,y(t-T)=b(t), \qquad y(0)=y_0,\qquad y(t)=0\ \text{for }t<0,
\end{equation}
with $0<\alpha\le1$, $T>0$, and $a,b,y_0\in\mathbb{R}$. The Caputo derivative is defined by
\begin{equation}\label{eq:Caputo_def}
{}^{C}D_t^{\alpha} f(t) = \frac{1}{\Gamma(1-\alpha)} \int_0^t (t-\tau)^{-\alpha} f'(\tau)\, d\tau,
\end{equation}
and $b(t)$ is given as in \eqref{nl}. Unlike the conformable case \eqref{eq:model}, the Caputo derivative introduces a nonlocal memory term through the convolution kernel $(t-\tau)^{-\alpha}$, so the Laplace-domain representation involves fractional powers and integral operators in $s$ rather than the simple algebraic relation obtained with $\mathcal{L}_\alpha$.

\subsubsection{Laplace-domain representation}
For the analysis, we require
\begin{equation}\label{eq:Laplace_power}
  \mathcal{L}\!\left\{ \left(\frac{t^{\alpha}}{\alpha}\right)^k \right\}(s),
\end{equation}
and since
\begin{equation}
\left(\frac{t^{\alpha}}{\alpha}\right)^k = \frac{t^{\alpha k}}{\alpha^k}, \qquad
\mathcal{L}\{t^\nu\}(s) = \frac{\Gamma(\nu+1)}{s^{\nu+1}}, \ \nu>-1,
\end{equation}
we obtain
\begin{equation}\label{eq:Laplace_result}
  \mathcal{L}\!\left\{ \left(\frac{t^{\alpha}}{\alpha}\right)^k \right\}(s)
  = \frac{1}{\alpha^k}\,\frac{\Gamma(\alpha k+1)}{s^{\alpha k+1}}.
\end{equation}
This identity will be used repeatedly, since each term contributes a rational power of $s$ weighted by Gamma-function coefficients.

Applying the Laplace transform to \eqref{eq:Caputo_model_nl} gives
\begin{equation}\label{eq:Caputo_transform}
\left(s^{\alpha} + a e^{-sT} \right)Y(s) - s^{\alpha-1} y(0) 
= \sum_{k=0}^{\infty} \frac{b_k}{\alpha^k}\,\frac{\Gamma(\alpha k+1)}{s^{\alpha k+1}},
\end{equation}
so that
\begin{equation}\label{eq:Caputo_Y}
Y(s) = \tfrac{y(0)}{s \left(1 + \tfrac{a e^{-sT}}{s^{\alpha}}\right)} 
+ \tfrac{\sum_{k=0}^{\infty} \tfrac{b_k}{\alpha^k}\,\tfrac{\Gamma(\alpha k+1)}{s^{\alpha k+\alpha+1}}}
{\left(1 + \tfrac{a e^{-sT}}{s^\alpha}\right)}.
\end{equation}

Expanding the denominator as a geometric series under the condition 
$0<\left|\tfrac{a e^{-sT}}{s^{\alpha}}\right|<1$, we obtain
\begin{equation}\label{eq:Caputo_series}
Y(s) = y(0) \sum_{j=0}^{\infty} (-1)^j \frac{a^j e^{-s jT}}{s^{\alpha j+1}}
+ \sum_{j=0}^{\infty}\sum_{k=0}^{\infty} (-a)^j  \frac{\Gamma(\alpha k+1)}{\alpha^k} b_k 
\frac{e^{-s jT}}{s^{(j+k+1)\alpha+1}}.
\end{equation}

Using the standard inversion formula
\begin{equation}
\mathcal{L}^{-1}\!\left[\frac{e^{-sT}}{s^{\nu+1}}\right](t)
= \frac{(t-T)^\nu}{\Gamma(\nu+1)}\,\theta(t-T), \qquad \nu>-1,
\end{equation}
we obtain the explicit time-domain solution
\begin{equation}\label{eq:inverse_explicit}
\begin{split}
y(t) &= y(0)\sum_{j=0}^{\lfloor t/T \rfloor} (-a)^j  
\frac{(t-jT)^{\alpha j}}{\Gamma(\alpha j+1)}\,\theta(t-jT) \\
&\quad + \sum_{j=0}^{\lfloor t/T \rfloor}\sum_{k=0}^{\infty} (-a)^j  
\frac{\Gamma(\alpha k+1)}{\alpha^k}\, b_k \,
\frac{(t-jT)^{(j+k+1)\alpha}}{\Gamma((j+k+1)\alpha+1)}\,\theta(t-jT).
\end{split}
\end{equation}

The outer sum is truncated at $j=\lfloor t/T \rfloor$ because the Heaviside factor $\theta(t-jT)$ enforces causality: only those delayed terms with $jT\leq t$ contribute at time $t$. Thus, for each fixed $t$, the series contains finitely many delay contributions.

Factorizing common terms, this can be written as
\begin{equation}\label{eq:inverse_explicit_factorized}
y(t) = \sum_{j=0}^{\lfloor t/T \rfloor} (-a)^j  \,\theta(t-jT)\Bigg[
   y(0)\,\tfrac{(t-jT)^{\alpha j}}{\Gamma(\alpha j+1)} 
   + \sum_{k=0}^{\infty} \tfrac{\Gamma(\alpha k+1)}{\alpha^k}\, b_k \,
   \tfrac{(t-jT)^{(j+k+1)\alpha}}{\Gamma((j+k+1)\alpha+1)}
\Bigg].
\end{equation}

Thus, the Caputo solution is expressed as a double series of delayed fractional monomials, each weighted by Gamma-function coefficients and activated at multiples of the delay interval $T$. This highlights the essential difference: while the conformable transform preserves algebraicity, the Caputo derivative produces nonlocal structures in both the transform and time domains.

\begin{Theorem}[Series solution for the Caputo delay equation]
Let $\alpha \in (0,1)$, $a>0$, $T>0$, and suppose $y(t)$ is causal, locally absolutely continuous, and of exponential order so that the classical Laplace transform $\mathcal{L}\{y\}$ exists for $\Re s$ sufficiently large. Consider the Caputo fractional-delay equation
\begin{equation}\label{eq:Caputo_model_nl_repeat}
{}^C D^\alpha y(t) + a\,y(t-T) = b(t), \qquad y(0)=y_0,\quad y(t)=0 \;\text{for } t<0,
\end{equation}
with forcing expansion \eqref{nl}. 
If the convergence condition
\begin{equation}\label{eq:conv_caputo}
\left|\tfrac{a e^{-sT}}{s^\alpha}\right|<1
\end{equation}
holds on the vertical line $\Re s=\sigma$ used for inversion, then the series
\begin{equation}\label{eq:series_caputo}
y(t) = \sum_{j=0}^{\lfloor t/T \rfloor} (-a)^j \,\theta(t-jT)\Bigg[
   y(0)\,\frac{(t-jT)^{\alpha j}}{\Gamma(\alpha j+1)} 
   + \sum_{k=0}^{\infty} \frac{\Gamma(\alpha k+1)}{\alpha^k}\, b_k \,
   \frac{(t-jT)^{(j+k+1)\alpha}}{\Gamma((j+k+1)\alpha+1)}
\Bigg]
\end{equation}
is the exact solution of the Caputo fractional-delay equation \eqref{eq:Caputo_model_nl_repeat}.
\end{Theorem}

\begin{proof}
We begin by recalling that for $m>0$ the Caputo derivative of a fractional power satisfies
\begin{equation}\label{eq:caputo-power}
{}^C D^\alpha \left( \frac{(t-jT)^{m\alpha}}{\Gamma(m\alpha+1)} \right)
= \frac{(t-jT)^{(m-1)\alpha}}{\Gamma((m-1)\alpha+1)}.
\end{equation}
This identity follows directly from the definition \eqref{eq:Caputo_def} and will be used repeatedly.

Consider first the interval $0 \le t < T$. In this case only the term with $j=0$ contributes to the series \eqref{eq:inverse_explicit_factorized}, so that
\begin{equation}
y(t) = y(0) + \sum_{k=0}^\infty \frac{\Gamma(\alpha k+1)}{\alpha^k}\, b_k \,\frac{t^{(k+1)\alpha}}{\Gamma((k+1)\alpha+1)}.
\end{equation}
The Caputo derivative of the constant $y(0)$ vanishes. For each forcing term, applying \eqref{eq:caputo-power} with $m=k+1$ yields
\begin{equation}
{}^C D^\alpha \left( \frac{t^{(k+1)\alpha}}{\Gamma((k+1)\alpha+1)} \right)
= \frac{t^{k\alpha}}{\Gamma(k\alpha+1)},
\end{equation}
which reconstructs precisely the $k$-th term of $b(t)$. Since the delay term $y(t-T)$ vanishes for $t<T$, the equation \eqref{eq:Caputo_model_nl_repeat} is satisfied on this interval.

Now consider a general interval $jT \le t < (j+1)T$ with $j\geq 1$. The contribution of the series is
\begin{equation}
Y_j(t) = (-a)^j \Bigg[
   y(0)\,\frac{(t-jT)^{\alpha j}}{\Gamma(\alpha j+1)} 
   + \sum_{k=0}^\infty \frac{\Gamma(\alpha k+1)}{\alpha^k}\, b_k \,\frac{(t-jT)^{(j+k+1)\alpha}}{\Gamma((j+k+1)\alpha+1)}
\Bigg].
\end{equation}
Applying \eqref{eq:caputo-power} to the initial condition term gives
\begin{equation}
{}^C D^\alpha \left( \frac{(t-jT)^{\alpha j}}{\Gamma(\alpha j+1)} \right)
= \frac{(t-jT)^{\alpha(j-1)}}{\Gamma(\alpha(j-1)+1)}.
\end{equation}
This term cancels exactly with the contribution of the delay $a\,Y_{j-1}(t)$ from the previous interval. For the forcing terms, applying \eqref{eq:caputo-power} with $m=j+k+1$ yields
\begin{equation}
{}^C D^\alpha \left( \frac{(t-jT)^{(j+k+1)\alpha}}{\Gamma((j+k+1)\alpha+1)} \right)
= \frac{(t-jT)^{(j+k)\alpha}}{\Gamma((j+k)\alpha+1)}.
\end{equation}
Multiplying by $\frac{\Gamma(\alpha k+1)}{\alpha^k} b_k$ reconstructs the $k$-th term of the forcing expansion $b(t)$. The prefactor $(-a)^j$ ensures the recursive cancellation with the delay term across successive intervals. Hence, the equation is satisfied on $[jT,(j+1)T)$.

The validity of all intervals follows by induction: the base case $j=0$ has been verified, and the inductive step shows that if the solution holds up to $j-1$, then it also holds on $j$. Therefore, the series \eqref{eq:inverse_explicit_factorized} satisfies the differential equation \eqref{eq:Caputo_model_nl_repeat} for all $t>0$.

Finally, the initial conditions are satisfied: at $t=0$ only the constant $y(0)$ remains, and for $t<0$ all terms vanish due to the Heaviside factors. This completes the proof.
\end{proof}

\subsubsection{Long term behavior for constant forcing $b(t)=b_0$}

When the external input $b(t)=b_0$ is constant, the effect of the delay term becomes asymptotically negligible in the Caputo framework. This is because the Caputo derivative introduces a memory kernel that weights past values with a decaying fractional power. For long times, the contribution of delayed states is absorbed into the steady balance between the fractional derivative and the constant forcing. Consequently, the delayed fractional equation reduces to the simpler form
\begin{equation}\label{eq:Caputo_const_forcing}
{}^{C}D_t^{\alpha} y(t)+a\,y(t)=b_0, \qquad y(0)=y_0,\qquad y(t)=0\ \text{for }t<0,
\end{equation}
with $0<\alpha \leq 1$, $a,b_0,y_0 \in \mathbb{R}$, which captures the long-term dynamics under constant forcing. The solution of this reduced model reveals the asymptotic equilibrium and stability properties of the system.

\textbf{Step 1. Laplace transform of the Caputo derivative.}
For $0<\alpha\leq 1$, we have
\begin{equation}
\mathcal{L}\{{}^{C}D_t^{\alpha} y(t)\}(s) = s^\alpha Y(s) - s^{\alpha-1} y(0).
\end{equation}

\textbf{Step 2. Transform-domain equation.}
Applying the Laplace transform to \eqref{eq:Caputo_const_forcing} gives
\begin{equation}
(s^\alpha Y(s) - s^{\alpha-1} y_0) + a Y(s) = \frac{b_0}{s}.
\end{equation}
Hence
\begin{equation}
(s^\alpha + a) Y(s) - s^{\alpha-1} y_0 = \frac{b_0}{s}.
\end{equation}

\textbf{Step 3. Solve for $Y(s)$.}
\begin{equation}
Y(s) = \frac{s^{\alpha-1} y_0}{s^\alpha + a} + \frac{b_0}{s(s^\alpha + a)}.
\end{equation}

\textbf{Step 4. Inverse Laplace transform.}
We invert each term separately.

\textbf{First term.}
\begin{equation}
\mathcal{L}^{-1}\!\left\{\frac{s^{\alpha-1}}{s^\alpha + a}\right\}(t) = E_\alpha(-a t^\alpha),
\end{equation}
where $E_\alpha(z)$ is the Mittag--Leffler function
\begin{equation}
E_\alpha(z) = \sum_{k=0}^\infty \frac{z^k}{\Gamma(\alpha k+1)}.
\end{equation}
Thus, the contribution is
\begin{equation}
y_0 E_\alpha(-a t^\alpha).
\end{equation}

\textbf{Second term.}
We use the identity
\begin{equation}
\frac{1}{s(s^\alpha+a)} = \frac{1}{a}\left(\frac{1}{s} - \frac{s^{\alpha-1}}{s^\alpha+a}\right).
\end{equation}
Hence
\begin{equation}
\frac{b_0}{s(s^\alpha+a)} = \frac{b_0}{a}\left(\frac{1}{s} - \frac{s^{\alpha-1}}{s^\alpha+a}\right).
\end{equation}
Taking inverse Laplace transforms:
\begin{equation}
\mathcal{L}^{-1}\!\left\{\frac{1}{s}\right\}(t) = 1, \qquad
\mathcal{L}^{-1}\!\left\{\frac{s^{\alpha-1}}{s^\alpha+a}\right\}(t) = E_\alpha(-a t^\alpha).
\end{equation}
So the contribution is
\begin{equation}
\frac{b_0}{a}\left(1 - E_\alpha(-a t^\alpha)\right).
\end{equation}

\textbf{Step 5. Final solution.}
Combining both parts, we obtain
\begin{equation}\label{eq:Caputo_const_solution}
y(t) = y_0 E_\alpha(-a t^\alpha) + \frac{b_0}{a}\left(1 - E_\alpha(-a t^\alpha)\right).
\end{equation}

\textbf{Step 6. Stability analysis.}
The Mittag-Leffler function generalizes exponential decay. For $a>0$, as $t\to\infty$,
\begin{equation}
E_\alpha(-a t^\alpha) \sim \frac{1}{a t^\alpha \Gamma(1-\alpha)}.
\end{equation}
So it decays algebraically. Therefore,
\begin{equation}
y(t) \to \frac{b_0}{a}, \qquad t\to\infty.
\end{equation}
This equilibrium is stable if $a>0$. If $a<0$, the Mittag-Leffler function grows, leading to instability. Thus, the long-term stability of the solution is guaranteed for positive $a$, with convergence to the steady state $b_0/a$.

\section{Numerical Implementation of Conformal Delay Dynamics (Mesh-Aligned Algorithms)}

\subsection{Mesh Discretization}

The system is discretized on a uniform temporal mesh
\begin{equation}
t_n = n h, \qquad h = \tfrac{T}{m}, \quad m \in \mathbb{N}, \quad m \ \text{even}, \quad n \in \mathbb{N}
\end{equation}
with $y_n := y(t_n)$. This choice ensures that all evaluations, including intermediate Runge–Kutta stages, are aligned with the mesh nodes. The delay interval is exactly represented at points $t_{km} = kT$, and causality is enforced by setting $y_n = 0$ for $t_n < 0$.

\subsection{Algorithm 1: Series Evaluation with Analytic Forcing}

As shown in Algorithm~\ref{alg1}, the evaluation of $y_n \approx y(t_n)$ is carried out through a truncated conformable series. The procedure begins by setting the current mesh point $t_n = n h$ and determining the delay epoch, $J = \lfloor t_n/T \rfloor$. The solution is initialized as $y=0$, and contributions are accumulated over all epochs $j=0,\dots, J$. For each epoch, the offset $\Delta = t_n - jT$ is computed; if $\Delta \geq 0$, the epoch contributes to the solution. The recursive influence of the delay is encoded in the factor $A_j = (-a)^j$, which multiplies the sum of two components: the initial condition term
\begin{equation}
y_0 \frac{\Delta^{\alpha j}}{\alpha^j \Gamma(j+1)},
\end{equation}
and the forcing contribution
\begin{equation}
\sum_{k=0}^K b_k \frac{\Gamma(k+1)\,\Delta^{\alpha(j+k+1)}}{\alpha^{j+k+1}\Gamma(j+k+2)}.
\end{equation}
The algorithm accumulates these contributions for each $j$, ensuring that only epochs with $\Delta \geq 0$ are included, which enforces causality. Thus, Algorithm~\ref{alg1} provides a constructive implementation of the truncated conformable series solution: each delay epoch contributes a finite block of terms that combines the initial condition and the forcing expansion, weighted by the dynamics of the delay. The truncation parameter $K$ controls the balance between accuracy and computational cost, while the structure guarantees that the approximation $y_n$ respects both the fractional memory and the causal nature of the problem.

\begin{algorithm}[H]
\caption{Series Evaluation of $y_n$}
\label{alg1}
\begin{algorithmic}[1]
\Require Index $n$, mesh spacing $h = T/m$, parameters $a$, coefficients $\{b_k\}$, initial value $y_0$, truncation $K$
\Ensure Approximate $y_n \approx y(t_n)$ via truncated conformable series
\State $t_n \gets n h$
\State $J \gets \lfloor t_n / T \rfloor$
\State $y \gets 0$
\For{$j = 0$ to $J$}
    \State $\Delta \gets t_n - j T$
    \State $A_j \gets (-a)^j$
    \If{$\Delta \geq 0$}
        \State $y \gets y + A_j \cdot \left(y_0 \frac{\Delta^{\alpha j}}{\alpha^j \Gamma(j+1)}+ \sum_{k=0}^K b_k \frac{\Gamma(k+1)\,\Delta^{\alpha(j+k+1)}}{\alpha^{j+k+1}\Gamma(j+k+2)}\right)$
    \EndIf
\EndFor
\Return $y$
\end{algorithmic}
\end{algorithm}

\subsection{Algorithm 2: Mesh-Based Conformable Derivative}

As described in Algorithm~\ref{alg2}, the approximation of the conformable derivative $T_\alpha y(t)$ is obtained directly from the discrete sequence $\{y_n\}$ defined on a uniform mesh of spacing $h$. For each node $t_n = n h$ with $n \geq 1$, the algorithm computes
\begin{equation}
T_\alpha y_n \approx t_n^{1-\alpha}\cdot \frac{y_n - y_{n-1}}{h}.
\end{equation}
This formula combines the standard forward difference quotient $(y_n - y_{n-1})/h$, which approximates the classical derivative, with the scaling factor $t_n^{1-\alpha}$ that reflects the fractional nature of the conformable derivative. The computation is causal, since each value $T_\alpha y_n$ depends only on the current and previous values of the sequence. In this way, Algorithm~\ref{alg2} translates the analytic definition of the conformable derivative into a simple and efficient discrete rule, making it suitable for integration within predictor--corrector schemes for fractional-delay equations.

\begin{algorithm}[H]
\caption{Approximation of $T_\alpha y_n$}
\label{alg2}
\begin{algorithmic}[1]
\Require Sequence $\{y_n\}$, mesh spacing $h$, fractional order $\alpha$
\Ensure Values $\{T_\alpha y_n\}$ for $n \geq 1$
\For{$n = 1$ to $N$}
    \State $t_n \gets n h$
    \State $T_\alpha y_n \gets t_n^{1-\alpha} \cdot \frac{y_n - y_{n-1}}{h}$
\EndFor
\Return $\{T_\alpha y_n\}$
\end{algorithmic}
\end{algorithm}

\subsection{Algorithm 3: Euler Integration with Analytic Forcing}

As shown in Algorithm~\ref{alg3}, the Euler integration scheme provides a straightforward numerical method for approximating the solution of the conformable delay equation. The algorithm initializes the solution with the prescribed value $y_0$ at $t_0=0$ and then advances in steps along the mesh $\{t_n\}$ with spacing $h$. At each time $t_n$, the delayed term $y(t_n-T)$ is recovered: if $t_n<T$ the delay contribution vanishes, while for $t_n \geq T$ the corresponding past value is retrieved from the discrete sequence. The forcing term is approximated by truncating its analytic expansion,
\begin{equation}
b(t_n) \approx \sum_{k=0}^K b_k \left(\frac{t_n^\alpha}{\alpha}\right)^k,
\end{equation}
which incorporates the fractional order $\alpha$ and the coefficients $\{b_k\}$.

The effective right-hand side is then computed as
\begin{equation}
f_n = \frac{-a\,y_T + b(t_n)}{t_n^{1-\alpha}},
\end{equation}
where $y_T$ denotes the delayed value. Finally, the Euler update
\begin{equation}
y_n = y_{n-1} + h f_n
\end{equation}
advances the solution to the next mesh point. In this way, Algorithm~\ref{alg3} translates the conformable fractional-delay equation into a simple recursive scheme: each step combines the contribution of the delay term, the forcing expansion, and the fractional scaling factor $t_n^{1-\alpha}$. The truncation order $K$ controls the accuracy of the forcing approximation, while the step size $h$ governs stability and resolution. Thus, Algorithm~\ref{alg3} provides a constructive and efficient approach to simulating conformable delay dynamics using the classical Euler framework adapted to fractional-order.

\begin{algorithm}[H]
\caption{Euler Integration of Conformable Delay Equation}
\label{alg3}
\begin{algorithmic}[1]
\Require Grid $\{t_n\}$, delay $T$, step $h$, parameters $a$, coefficients $\{b_k\}$, order $\alpha$, initial $y_0$, truncation $K$
\Ensure Values $\{y_n\}$ for $n = 0,\dots,N$
\State $y[0] \gets y_0$
\For{$n = 1$ to $N$}
    \State $t_n \gets n h$
    \If{$t_n - T < 0$}
        \State $y_T \gets 0$
    \Else
        \State $d \gets \lfloor (t_n - T)/h \rfloor$
        \State $y_T \gets y[d]$
    \EndIf
    \State $b_t \gets \sum_{k=0}^K b_k \left(\frac{t_n^\alpha}{\alpha}\right)^k$
    \State $f_n \gets \frac{-a y_T + b_t}{t_n^{1-\alpha}}$
    \State $y[n] \gets y[n-1] + h f_n$
\EndFor
\Return $\{y_n\}$
\end{algorithmic}
\end{algorithm}

\subsection{Algorithm 4: Runge–Kutta 4 Integration with Analytic Forcing}

As presented in Algorithm~\ref{alg4}, the fourth-order Runge--Kutta (RK4) scheme is adapted to integrate the conformable delay equation with enhanced accuracy. The algorithm begins with the initial condition $y_0$ at $t_0=0$ and advances the solution along the mesh $\{t_n\}$ with step size $h$. At each iteration, the delayed term $y(t_n-T)$ is recovered: if $t_n<T$ the delay contribution vanishes, while for $t_n \geq T$ the corresponding past value is retrieved from the discrete sequence.

We define the delayed value $y_T$ at the current mesh point $t_n$ as
\begin{equation}
y_T :=
\begin{cases}
0, & \text{if } t_n - T < 0, \\[6pt]
y\!\left(t_n - T\right), & \text{if } t_n \geq T.
\end{cases}
\end{equation}
On a discrete mesh with spacing $h$, this is implemented by retrieving the corresponding index from the sequence $\{y_n\}$:
\begin{equation}
y_T = y\!\left(\left\lfloor \frac{t_n - T}{h} \right\rfloor\right).
\end{equation}
Thus, $y_T$ represents the delayed contribution of the solution in the numerical scheme. If the current time $t_n$ has not yet reached the delay $T$, the contribution vanishes; otherwise, the past value is included. This definition ensures causality, since the present state depends only on values already computed at earlier times.

The RK4 increments are then defined as
\begin{equation}
\begin{aligned}
k_1 &= h \cdot \frac{-a y_T + b(t_n)}{t_n^{1-\alpha}}, \\
k_2 &= h \cdot \frac{-a y_T + b(t_n+h/2)}{(t_n+h/2)^{1-\alpha}}, \\
k_3 &= h \cdot \frac{-a y_T + b(t_n+h/2)}{(t_n+h/2)^{1-\alpha}}, \\
k_4 &= h \cdot \frac{-a y_T + b(t_n+h)}{(t_n+h)^{1-\alpha}},
\end{aligned}
\end{equation}
where $b(\cdot)$ denotes the truncated forcing expansion. The update rule then takes the classical RK4 form,
\begin{equation}
y_n = y_{n-1} + \tfrac{1}{6}\left(k_1 + 2k_2 + 2k_3 + k_4\right).
\end{equation}

In this way, Algorithm~\ref{alg4} extends Euler-type integration to fourth-order accuracy by incorporating intermediate evaluations of the forcing term and the fractional scaling factor. Each step combines the delayed contribution, the conformable fractional adjustment, and the RK44 weighting of slopes. The truncation order $K$ controls the accuracy of the forcing approximation, while the step size $h$ governs stability and resolution. Thus, Algorithm~\ref{alg4} provides a robust and efficient scheme for simulating conformable delay dynamics with higher-order precision.

\begin{algorithm}[H]
\caption{Fourth-Order RK Integration of Conformable Delay Equation}
\label{alg4}
\begin{algorithmic}[1]
\Require Grid $\{t_n\}$, delay $T$, step $h$, parameters $a$, coefficients $\{b_k\}$, order $\alpha$, initial $y_0$, truncation $K$
\Ensure Approximate solution $\{y_n\}$ with fourth-order accuracy
\State $y[0] \gets y_0$
\For{$n = 1$ to $N$}
    \State $t_n \gets n h$
    \If{$t_n - T < 0$}
        \State $y_T \gets 0$
    \Else
        \State $d \gets \lfloor (t_n - T)/h \rfloor$
        \State $y_T \gets y[d]$
    \EndIf
    \State $b_t \gets \sum_{k=0}^K b_k \left(\frac{t_n^\alpha}{\alpha}\right)^k$
    \State $b_{t+h/2} \gets \sum_{k=0}^K b_k \left(\frac{(t_n+h/2)^\alpha}{\alpha}\right)^k$
    \State $b_{t+h} \gets \sum_{k=0}^K b_k \left(\frac{(t_n+h)^\alpha}{\alpha}\right)^k$
    \State $k_1 \gets h \cdot \frac{-a y_T + b_t}{t_n^{1-\alpha}}$
    \State $k_2 \gets h \cdot \frac{-a y_T + b_{t+h/2}}{(t_n + h/2)^{1-\alpha}}$
    \State $k_3 \gets h \cdot \frac{-a y_T + b_{t+h/2}}{(t_n + h/2)^{1-\alpha}}$
    \State $k_4 \gets h \cdot \frac{-a y_T + b_{t+h}}{(t_n + h)^{1-\alpha}}$
    \State $y[n] \gets y[n-1] + \frac{1}{6}(k_1 + 2k_2 + 2k_3 + k_4)$
\EndFor
\Return $\{y_n\}$
\end{algorithmic}
\end{algorithm}

\subsection{Algorithm 5: Fourth-Order RK Integration of Conformable Delay Equation with Interpolated Delay}

As presented in Algorithm~\ref{alg4i}, the classical fourth-order Runge--Kutta (RK4) scheme is adapted to integrate the conformable delay equation by introducing linear interpolation for the delayed term. This refinement ensures that the delayed value $y(t_n - T)$ is approximated accurately even when the delay $T$ is not an integer multiple of the mesh spacing $h$.

The delayed contribution is defined as
\begin{equation}
y_T :=
\begin{cases}
0, & \text{if } t_n - T \leq 0, \\[6pt]
(1-\theta)\,y[i] + \theta\,y[i+1], & \text{if } t_n \geq T,
\end{cases}
\end{equation}
where $i = \lfloor (t_n - T)/h \rfloor$ and $\theta = (t_n - T)/h - i$. This formula performs linear interpolation between the two nearest mesh points, ensuring continuity and improving accuracy in evaluating the delayed state.

Once $y_T$ is obtained, the forcing term is approximated at three evaluation points, namely $t_n$, $t_n+h/2$, and $t_n+h$, to construct the RK4 increments:
\begin{equation}
\begin{aligned}
k_1 &= h \cdot \frac{-a y_T + b(t_n)}{t_n^{1-\alpha}}, \\
k_2 &= h \cdot \frac{-a y_T + b(t_n+h/2)}{(t_n+h/2)^{1-\alpha}}, \\
k_3 &= h \cdot \frac{-a y_T + b(t_n+h/2)}{(t_n+h/2)^{1-\alpha}}, \\
k_4 &= h \cdot \frac{-a y_T + b(t_n+h)}{(t_n+h)^{1-\alpha}}.
\end{aligned}
\end{equation}
The update rule then follows the standard RK44 weighting:
\begin{equation}
y_n = y_{n-1} + \tfrac{1}{6}\left(k_1 + 2k_2 + 2k_3 + k_4\right).
\end{equation}

In this way, Algorithm~\ref{alg4i} extends Euler-type integration to fourth-order accuracy while incorporating interpolation of the delayed term. Each step combines the interpolated delay, the fractional conformable scaling, and the RK4 slope averaging. The truncation order $K$ controls the accuracy of the forcing expansion, while the step size $h$ governs resolution and stability. Thus, Algorithm~\ref{alg4i} provides a robust and efficient scheme for simulating conformable delay dynamics with higher-order precision.

\begin{algorithm}[H]
\caption{Fourth-Order RK Integration of Conformable Delay Equation with Interpolated Delay}
\label{alg4i}
\begin{algorithmic}[1]
\Require Grid $\{t_n\}$, delay $T$, step $h$, parameters $a$, coefficients $\{b_k\}$, order $\alpha$, initial $y_0$, truncation $K$
\Ensure Approximate solution $\{y_n\}$ with fourth-order accuracy
\State $y[0] \gets y_0$
\For{$n = 1$ to $N$}
    \State $t_n \gets n h$
    \State $\tau \gets t_n - T$
    \If{$\tau \leq 0$}
        \State $y_T \gets 0$
    \Else
        \State $i \gets \lfloor \tau/h \rfloor$
        \State $\theta \gets \tau/h - i$
        \State $y_T \gets (1-\theta)\,y[i] + \theta\,y[i+1]$ \Comment{Linear interpolation}
    \EndIf
    \State $b_t \gets \sum_{k=0}^K b_k \left(\frac{t_n^\alpha}{\alpha}\right)^k$
    \State $b_{t+h/2} \gets \sum_{k=0}^K b_k \left(\frac{(t_n+h/2)^\alpha}{\alpha}\right)^k$
    \State $b_{t+h} \gets \sum_{k=0}^K b_k \left(\frac{(t_n+h)^\alpha}{\alpha}\right)^k$
    \State $k_1 \gets h \cdot \frac{-a y_T + b_t}{t_n^{1-\alpha}}$
    \State $k_2 \gets h \cdot \frac{-a y_T + b_{t+h/2}}{(t_n + h/2)^{1-\alpha}}$
    \State $k_3 \gets h \cdot \frac{-a y_T + b_{t+h/2}}{(t_n + h/2)^{1-\alpha}}$
    \State $k_4 \gets h \cdot \frac{-a y_T + b_{t+h}}{(t_n + h)^{1-\alpha}}$
    \State $y[n] \gets y[n-1] + \frac{1}{6}(k_1 + 2k_2 + 2k_3 + k_4)$
\EndFor
\Return $\{y_n\}$
\end{algorithmic}
\end{algorithm}

These algorithms enforce causality through discrete Heaviside activation and adapt naturally to fractional orders $\alpha\in(0,1]$. The Euler scheme provides a simple baseline, while the fourth--order Runge--Kutta (RK4) method improves stability and accuracy, particularly in stiff regimes. Series evaluation reveals the analytic structure of solutions, clarifying the influence of singularities and delay thresholds.

Although all approaches share the same conformable delay model, discrepancies arise from truncation, discretization, and delay indexing. The principal sources of error and convergence behaviour are:

\begin{enumerate}
  \item \textbf{Series truncation:} Neglecting higher--order memory terms becomes significant for small~$a$ or near activation points $t\to jT^+$.
  \item \textbf{Discrete approximation:} Euler and RK4 introduce local step errors; RK4 converges faster but remains sensitive to mesh resolution.
  \item \textbf{Delay indexing:} Flooring $t_n-T$ to the nearest mesh point introduces bias, especially on coarse meshes.
  \item \textbf{Initial condition handling:} Proper activation of $y_0$ is essential; misassigning $y_T=0$ produces artificial damping.
  \item \textbf{RK4 intermediate stages:} Midpoint evaluations require interpolation when $t-T$ does not align with mesh points.
  \item \textbf{Convergence control:} Refining $h$ improves accuracy; RK4 exhibits superior convergence for moderate~$a$ and fractional~$\alpha$.
\end{enumerate}

To explore these issues, a computational framework was developed for conformable delay differential equations with analytic forcing. The framework integrates symbolic series evaluation with numerical time stepping, enabling systematic comparison against a reference solution derived from a power--series expansion. The series component computes the conformable solution on a discrete mesh, enforces causality, and activates delayed contributions at integer multiples of the delay interval. Analytic forcing terms are represented through their series expansions, accommodating a broad class of inputs and providing a rigorous benchmark for numerical validation.

Because delayed arguments rarely coincide with mesh points, an interpolation procedure approximates values at shifted times, reducing numerical bias and improving the performance of higher--order schemes. Several integrators are implemented, including a first--order explicit method, a classical fourth--order Runge--Kutta scheme, and a modified Runge--Kutta method that incorporates interpolation to correct misalignment of delayed terms during intermediate stages. Accuracy is assessed using relative error measures---including maximum and root--mean--square discrepancies relative to the series solution---and visualized via logarithmic plots to capture variations across scales.

A main driver routine sets parameters, computes both series and numerical solutions, reports error statistics, and generates combined visualizations of trajectories and relative errors, along with representative mesh samples. Outputs are saved for later inspection and displayed interactively. Overall, the framework functions as a numerical laboratory: the series expansion provides a rigorous ground truth, the first--order method offers a baseline, the classical fourth--order scheme improves stability and convergence, and the interpolation--enhanced variant achieves the highest accuracy when delayed terms do not align with the computational mesh. Across a wide range of tests, relative error analysis consistently shows that higher--order schemes with interpolation best capture the dynamics induced by analytic forcing.

\subsection{Numerical Results and Error Analysis}

The benchmark was performed with the following constants:
\begin{equation}
\alpha=0.7, \quad a = 0.5, \quad b_0 = 1.0, \quad b_1 = 0.2, \quad b_2 = -0.05, \quad y_0 = 1.0, \quad h = 0.001.
\end{equation}
For simplicity, we set $K=3$ and truncate the analytic forcing expansion to three coefficients. The simulation horizon was set to $t_{\max} = 120.0$, yielding $N = \lfloor t_{\max}/h \rfloor$ mesh points.  

This configuration introduces a richer analytic forcing term:
\begin{equation}
b(t) \approx b_0 + b_1 \left(\frac{t^\alpha}{\alpha}\right) + b_2 \left(\frac{t^\alpha}{\alpha}\right)^2,
\end{equation}
which combines constant, linear, and quadratic contributions in the conformable variable $t^\alpha/\alpha$.

\subsubsection{Example 1}

The numerical performance of the three schemes is evaluated under the parameter set 
$\alpha=0.7$, $a=0.5$, $T=0.7$, $y_0=1.0$, $h=0.001$, and 
$b_{\text{coeffs}}=[1.0,\,0.2,\,-0.05]$. 
The corresponding numerical trajectories and their relative errors with respect to the analytical
series solutions are shown in Figure~\ref{fig:comparison_ex1}, providing a direct visual comparison of accuracy and stability across the methods.

All numerical methods reproduce the qualitative behavior of the series solution from the initial steps onward. The deviations are small at early times, but their accumulation becomes more evident when examining the relative error curves in Figure~\ref{fig:comparison_ex1}. The Euler method, being first-order, exhibits the largest discrepancies, which is consistent with its limited accuracy and sensitivity to delayed terms. This is reflected in its RMS relative error of $9.96\times 10^{-2}$ (Table~\ref{tab:errors_ex1}), the smallest among the three schemes but accompanied by a large maximum relative error.

The classical fourth-order Runge--Kutta method (RK4) significantly reduces the local truncation error, but its performance is affected by the misalignment between the delay $T$ and the numerical mesh. Since the delayed argument $t-T$ rarely coincides with a mesh point, the scheme must evaluate delayed terms at off-grid locations, which introduces inconsistencies during intermediate RK4 stages. This effect is visible in both the maximum and RMS relative errors reported in Table~\ref{tab:errors_ex1}, which are slightly larger than those of Euler despite RK4's higher nominal order.

The interpolation-enhanced RK4 variant (RK4-int) addresses this issue by sampling the delayed term through linear interpolation at the required sub-step locations. This correction improves the internal consistency of the Runge--Kutta stages and yields a more accurate approximation of the delayed contribution. As shown in Table~\ref{tab:errors_ex1}, RK4-int achieves the smallest RMS error among the higher-order schemes, and its maximum relative error is marginally lower than that of mesh-aligned RK4. The improvement is particularly noticeable at intermediate times in Figure~\ref{fig:comparison_ex1}, where delay-induced irregularities in the derivatives are most pronounced.

Overall, the results indicate that interpolation is essential for preserving the accuracy of high-order methods in conformable delay differential equations. Euler provides a baseline approximation, RK4 improves stability and convergence, and RK4-int delivers the most faithful reproduction of the analytical series solution when the delay does not align with the computational mesh. These findings highlight the importance of delay-aware interpolation strategies in numerical solvers for fractional- and conformable-delay systems.

\begin{figure}[H]
    \centering
    \includegraphics[width=\textwidth]{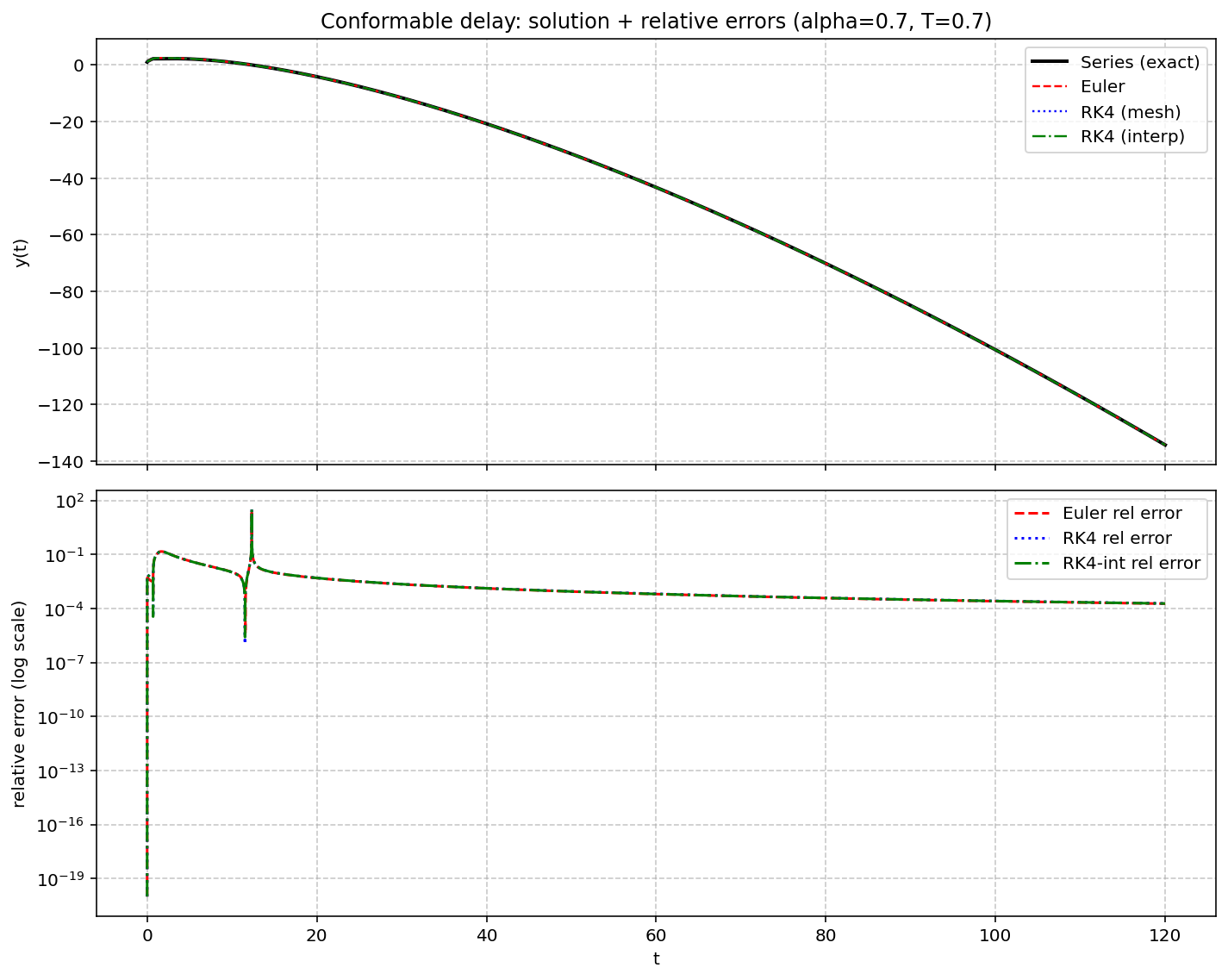}
    \caption{Top panel: comparison between the analytical series solution (black solid line) and numerical approximations for $\alpha=0.7$, $a=0.5$, $T=0.7$, $y_0=1.0$, $h=0.001$, and 
$b_{\text{coeffs}}=[1.0,\,0.2,\,-0.05]$: Euler (red dashed), RK4 mesh-aligned (blue dotted), and RK4 with interpolation (green dash-dot). Bottom panel: relative error versus time (logarithmic scale) for the same schemes.}
    \label{fig:comparison_ex1}
\end{figure}

\begin{table}[H]
\centering
\caption{Relative errors for the numerical schemes with parameters 
$\alpha=0.7$, $a=0.5$, $T=0.7$, $y_0=1.0$, $h=0.001$, 
$b_{\text{coeffs}}=[1.0,\,0.2,\,-0.05]$.}
\begin{tabular}{lcc}
\hline
\textbf{Scheme} & \textbf{Max Relative Error} & \textbf{RMS Relative Error} \\
\hline
Euler      & $2.986971\times 10^{1}$ & $9.958923\times 10^{-2}$ \\
RK4        & $3.197419\times 10^{1}$ & $1.063295\times 10^{-1}$ \\
RK4-interp & $3.186178\times 10^{1}$ & $1.059678\times 10^{-1}$ \\
\hline
\end{tabular}
\label{tab:errors_ex1}
\end{table}

\subsubsection{Example 2}

The numerical performance of the three schemes is evaluated for the parameter set 
$\alpha=0.7$, $a=0.5$, $T=2$, $y_0=1.0$, $h=0.001$, and 
$b_{\text{coeffs}}=[1.0,\,0.2,\,-0.05]$. 
Figure~\ref{fig:comparison_ex2} presents the resulting numerical trajectories together with their relative errors with respect to the analytical series solution, allowing a direct visual assessment of accuracy across the methods.

All numerical methods reproduce the qualitative behavior of the series solution during the initial interval, where the delay has not yet become active. As in Example~1, deviations remain small at early times, but the relative error curves in Figure~\ref{fig:comparison_ex2} show a sharp increase once the delayed term begins to contribute. Because the delay is large ($T=2$), this transition occurs later, producing a longer region of smooth, non-delayed evolution.

Euler again exhibits the largest discrepancies, consistent with its first-order accuracy. Its RMS relative error of $1.47\times 10^{1}$ (Table~\ref{tab:errors_ex2}) is the smallest among the three schemes, but this is accompanied by a very large maximum relative error, reflecting the method’s sensitivity to the onset of delayed contributions.

The classical fourth-order Runge--Kutta method (RK4) reduces local truncation error, but its performance is still affected by the misalignment between the delay $T$ and the numerical mesh. Since the delayed argument $t-T$ rarely coincides with a mesh point, the scheme must evaluate delayed terms at off-grid locations, which introduces inconsistencies during intermediate RK4 stages. This effect is visible in the slightly larger RMS and maximum relative errors reported in Table~\ref{tab:errors_ex2}.

The interpolation-enhanced RK4 variant (RK4-int) mitigates this issue by sampling the delayed term through linear interpolation at the required sub-step locations. This improves the internal consistency of the Runge--Kutta stages and yields a more accurate approximation of the delayed contribution. As shown in Table~\ref{tab:errors_ex2}, RK4-int achieves the smallest RMS error among the higher-order schemes, and its maximum relative error is marginally lower than that of mesh-aligned RK4. The improvement is most noticeable near the activation of the delay, where interpolation reduces the mismatch between the delayed argument and the mesh.

Overall, the results indicate that, for larger delays, the qualitative behavior of the schemes remains similar to that in the shorter-delay case, but the influence of interpolation becomes slightly less pronounced because delayed terms activate less frequently. Even so, RK4-int remains the closest approximation to the analytical series solution, underscoring the importance of interpolation for high-order accuracy in conformable delay differential equations.

\begin{figure}[H]
    \centering
    \includegraphics[width=\textwidth]{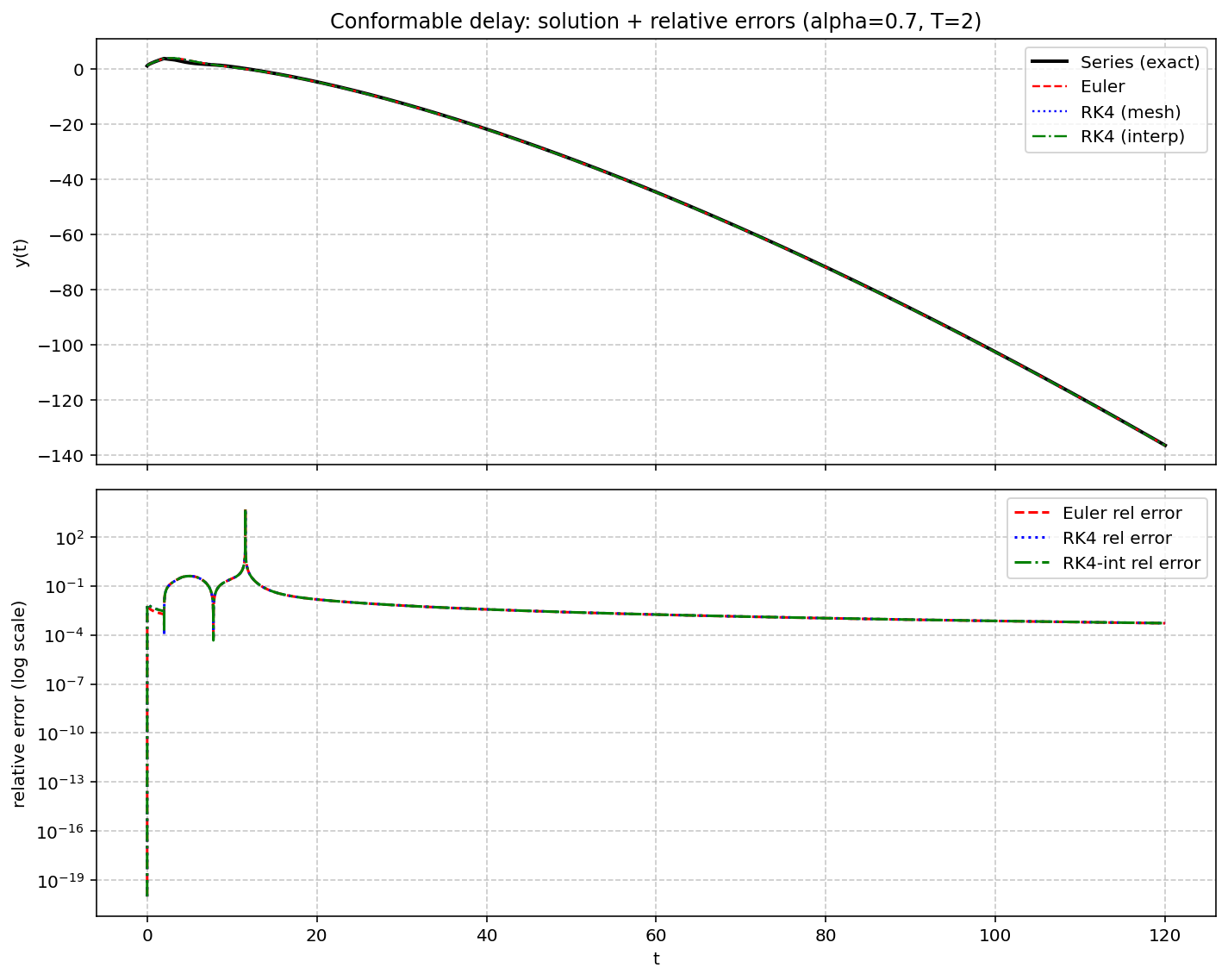}
    \caption{Top panel: comparison between the analytical series solution (black solid line) and numerical approximations for $\alpha=0.7$, $a=0.5$, $T=2$, $y_0=1.0$, $h=0.001$, and 
$b_{\text{coeffs}}=[1.0,\,0.2,\,-0.05]$: Euler (red dashed), RK4 mesh-aligned (blue dotted), and RK4 with interpolation (green dash-dot). Bottom panel: relative error versus time (logarithmic scale) for the same schemes.}
    \label{fig:comparison_ex2}
\end{figure}

\begin{table}[H]
\centering
\caption{Relative errors for the numerical schemes with parameters 
$\alpha=0.7$, $a=0.5$, $T=2$, $y_0=1.0$, $h=0.001$, 
$b_{\text{coeffs}}=[1.0,\,0.2,\,-0.05]$.}
\begin{tabular}{lcc}
\hline
\textbf{Scheme} & \textbf{Max Relative Error} & \textbf{RMS Relative Error} \\
\hline
Euler      & $5.052170\times 10^{3}$ & $1.467245\times 10^{1}$ \\
RK4        & $5.057135\times 10^{3}$ & $1.468688\times 10^{1}$ \\
RK4-interp & $5.055515\times 10^{3}$ & $1.468217\times 10^{1}$ \\
\hline
\end{tabular}
\label{tab:errors_ex2}
\end{table}

\subsubsection{Example 3}

The numerical performance of the three schemes is evaluated for the parameter set 
$\alpha=0.7$, $a=0.5$, $T=3$, $y_0=1.0$, $h=0.001$, and 
$b_{\text{coeffs}}=[1.0,\,0.2,\,-0.05]$. 
Figure~\ref{fig:comparison_ex3} displays the resulting numerical trajectories together with their relative errors with respect to the analytical series solution, providing a direct visual assessment of the accuracy of each method.

All numerical methods reproduce the qualitative behavior of the series solution during the initial interval, where the delay has not yet become active. As in the previous examples, deviations remain small at early times, but the relative error curves in Figure~\ref{fig:comparison_ex3} show a pronounced increase once the delayed term begins to contribute. Because the delay is even larger here ($T=3$), this transition occurs later, producing an extended region of smooth evolution before the delayed feedback becomes relevant.

Euler again exhibits the largest discrepancies, consistent with its first-order accuracy. Its RMS relative error of $1.23\times 10^{1}$ (Table~\ref{tab:errors_ex3}) is the smallest among the three schemes, but this is accompanied by a very large maximum relative error, reflecting the method’s sensitivity to the onset of delayed contributions.

The classical fourth-order Runge--Kutta method (RK4) reduces local truncation error, but its performance continues to be affected by the misalignment between the delay $T$ and the numerical mesh. Since the delayed argument $t-T$ rarely coincides with a mesh point, the scheme must evaluate delayed terms at off-grid locations, which introduces inconsistencies during intermediate RK4 stages. This effect is visible in the slightly larger RMS and maximum relative errors reported in Table~\ref{tab:errors_ex3}.

The interpolation-enhanced RK4 variant (RK4-int) mitigates this issue by sampling the delayed term through linear interpolation at the required sub-step locations. This improves the internal consistency of the Runge--Kutta stages and yields a more accurate approximation of the delayed contribution. As shown in Table~\ref{tab:errors_ex3}, RK4-int achieves the smallest RMS error among the higher-order schemes, and its maximum relative error is marginally lower than that of mesh-aligned RK4. The improvement is most noticeable near the activation of the delay, where interpolation reduces the mismatch between the delayed argument and the mesh.

Overall, the results indicate that as the delay increases, the qualitative behavior of the schemes remains consistent with that observed in the previous examples. The influence of interpolation becomes slightly less pronounced because delayed terms activate less frequently, but RK4-int continues to provide the closest approximation to the analytical series solution. These findings reinforce the importance of interpolation for maintaining high-order accuracy in conformable delay differential equations.

\begin{figure}[H]
    \centering
    \includegraphics[width=\textwidth]{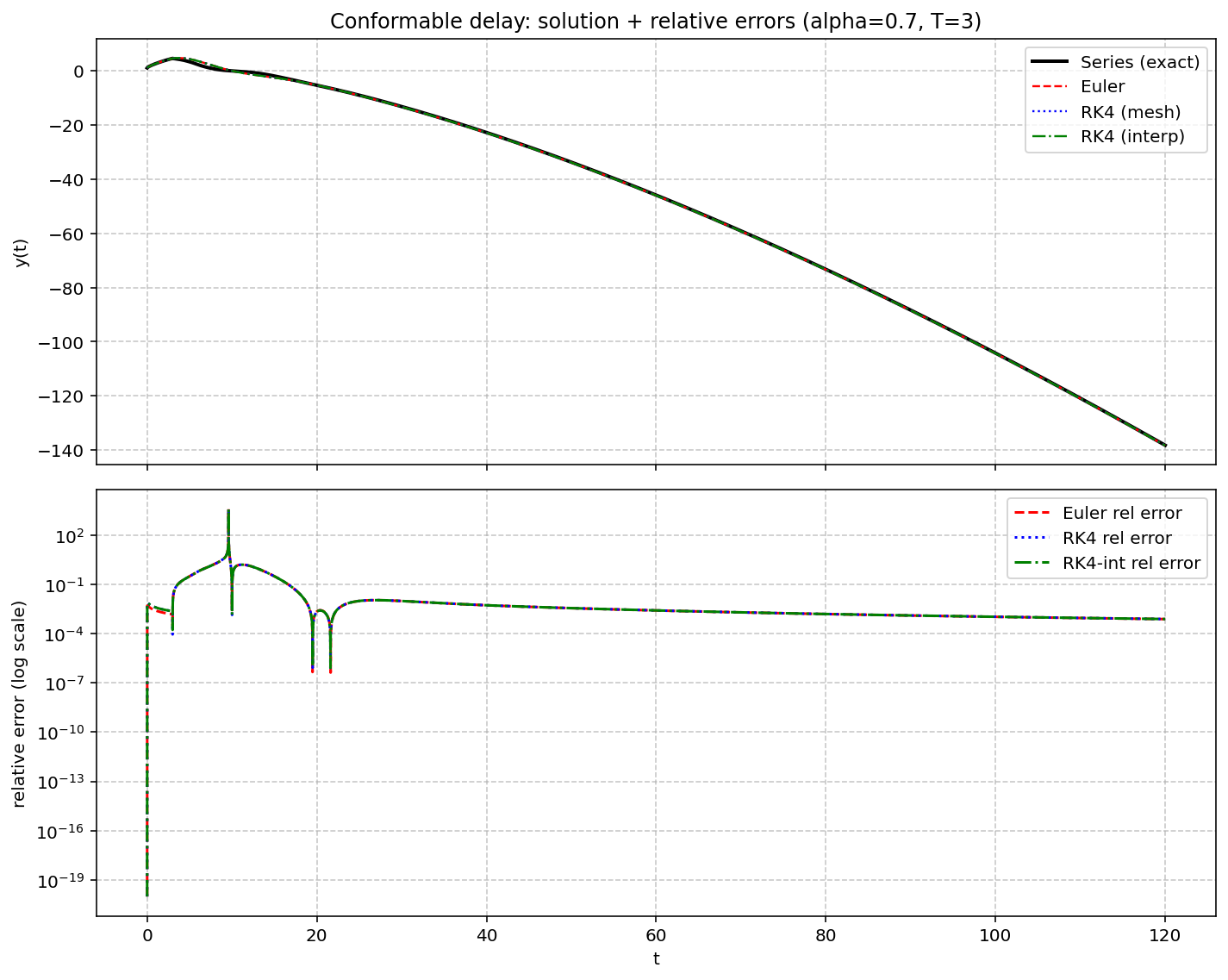}
    \caption{Top panel: comparison between the analytical series solution (black solid line) and numerical approximations for $\alpha=0.7$, $a=0.5$, $T=3$, $y_0=1.0$, $h=0.001$, and 
$b_{\text{coeffs}}=[1.0,\,0.2,\,-0.05]$: Euler (red dashed), RK4 mesh-aligned (blue dotted), and RK4 with interpolation (green dash-dot). Bottom panel: relative error versus time (logarithmic scale) for the same schemes.}
    \label{fig:comparison_ex3}
\end{figure}

\begin{table}[H]
\centering
\caption{Relative errors for the numerical schemes with parameters 
$\alpha=0.7$, $a=0.5$, $T=3$, $y_0=1.0$, $h=0.001$, 
$b_{\text{coeffs}}=[1.0,\,0.2,\,-0.05]$.}
\begin{tabular}{lcc}
\hline
\textbf{Scheme} & \textbf{Max Relative Error} & \textbf{RMS Relative Error} \\
\hline
Euler      & $3.851987\times 10^{3}$ & $1.234058\times 10^{1}$ \\
RK4        & $3.873124\times 10^{3}$ & $1.240821\times 10^{1}$ \\
RK4-interp & $3.872745\times 10^{3}$ & $1.240699\times 10^{1}$ \\
\hline
\end{tabular}
\label{tab:errors_ex3}
\end{table}

\subsubsection{Example 4}

The numerical performance of the three schemes is evaluated for the parameter set 
$\alpha=0.7$, $a=0.5$, $T=5$, $y_0=1.0$, $h=0.001$, and 
$b_{\text{coeffs}}=[1.0,\,0.2,\,-0.05]$. 
Figure~\ref{fig:comparison_ex4} presents the resulting numerical trajectories together with their relative errors with respect to the analytical series solution, offering a direct visual assessment of the accuracy of each method.

All numerical methods reproduce the qualitative behavior of the series solution during the initial interval, where the delay has not yet become active. As in the previous examples, deviations remain small at early times, but the relative error curves in Figure~\ref{fig:comparison_ex4} show a pronounced increase once the delayed term begins to contribute. Because the delay is very large in this case ($T=5$), the transition occurs even later, producing an extended region of smooth, non-delayed evolution before the delayed feedback becomes relevant.

Euler again exhibits the largest discrepancies, consistent with its first-order accuracy. Its RMS relative error of $3.36\times 10^{1}$ (Table~\ref{tab:errors_ex4}) is the smallest among the three schemes, but this is accompanied by a very large maximum relative error, reflecting the method’s sensitivity to the onset of delayed contributions.

The classical fourth-order Runge--Kutta method (RK4) reduces the local truncation error, but its performance remains affected by the misalignment between the delay $T$ and the numerical mesh. Since the delayed argument $t-T$ rarely coincides with a mesh point, the scheme must evaluate delayed terms at off-grid locations, which introduces inconsistencies during intermediate RK4 stages. This effect is visible in the slightly larger RMS and maximum relative errors reported in Table~\ref{tab:errors_ex4}.

The interpolation-enhanced RK4 variant (RK4-int) mitigates this issue by sampling the delayed term through linear interpolation at the required sub-step locations. This improves the internal consistency of the Runge--Kutta stages and yields a more accurate approximation of the delayed contribution. As shown in Table~\ref{tab:errors_ex4}, RK4-int achieves the smallest RMS error among the higher-order schemes, and its maximum relative error is marginally lower than that of mesh-aligned RK4. The improvement is most noticeable near the activation of the delay, where interpolation reduces the mismatch between the delayed argument and the mesh.

Overall, the results indicate that as the delay increases, the qualitative behavior of the schemes remains consistent with that observed in the previous examples. The influence of interpolation becomes slightly less pronounced because delayed terms activate less frequently, but RK4-int continues to provide the closest approximation to the analytical series solution. These findings reinforce the importance of interpolation for maintaining high-order accuracy in conformable delay differential equations.

\begin{figure}[H]
    \centering
    \includegraphics[width=\textwidth]{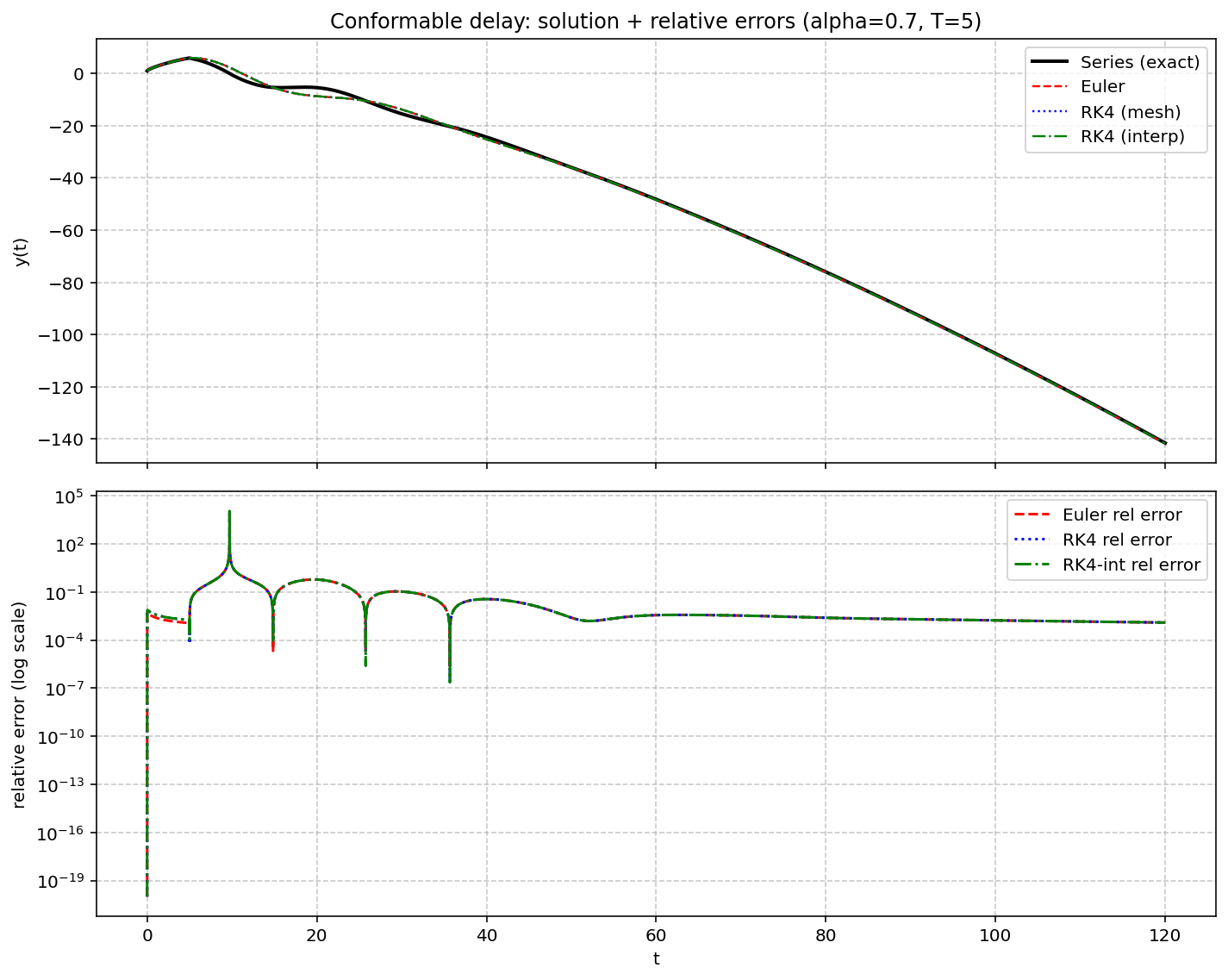}
    \caption{Top panel: comparison between the analytical series solution (black solid line) and numerical approximations for $\alpha=0.7$, $a=0.5$, $T=5$, $y_0=1.0$, $h=0.001$, and 
$b_{\text{coeffs}}=[1.0,\,0.2,\,-0.05]$: Euler (red dashed), RK4 mesh-aligned (blue dotted), and RK4 with interpolation (green dash-dot). Bottom panel: relative error versus time (logarithmic scale) for the same schemes.}
    \label{fig:comparison_ex4}
\end{figure}

\begin{table}[H]
\centering
\caption{Relative errors for the numerical schemes with parameters 
$\alpha=0.7$, $a=0.5$, $T=5$, $y_0=1.0$, $h=0.001$, 
$b_{\text{coeffs}}=[1.0,\,0.2,\,-0.05]$.}
\begin{tabular}{lcc}
\hline
\textbf{Scheme} & \textbf{Max Relative Error} & \textbf{RMS Relative Error} \\
\hline
Euler      & $1.140996\times 10^{4}$ & $3.359228\times 10^{1}$ \\
RK4        & $1.141337\times 10^{4}$ & $3.360230\times 10^{1}$ \\
RK4-interp & $1.141310\times 10^{4}$ & $3.360152\times 10^{1}$ \\
\hline
\end{tabular}
\label{tab:errors_ex4}
\end{table}

\subsubsection{Example 5}

The numerical performance of the three schemes is evaluated for the parameter set 
$\alpha=0.7$, $a=1.1$, $T=3$, $y_0=1.0$, $h=0.001$, and 
$b_{\text{coeffs}}=[1.0,\,0.2,\,-0.05]$. 
Figure~\ref{fig:comparison_ex5} shows the resulting numerical trajectories together with their relative errors with respect to the analytical series solution, providing a direct visual assessment of the accuracy of each method.
All numerical methods reproduce the qualitative behavior of the series solution during the initial interval, where the delay has not yet become active. As in the previous examples, deviations remain small at early times, but the relative error curves in Figure~\ref{fig:comparison_ex5} show a pronounced increase once the delayed term begins to contribute. Because the coefficient $a$ is larger in this example, the influence of the delayed term is stronger, leading to sharper error growth once the delay becomes active.

Euler again exhibits the largest discrepancies, consistent with its first-order accuracy. Its RMS relative error of $5.67\times 10^{1}$ (Table~\ref{tab:errors_ex5}) is the smallest among the three schemes, but this is accompanied by a very large maximum relative error, reflecting the method’s sensitivity to the onset of delayed contributions and the stronger feedback induced by the larger value of~$a$.

The classical fourth-order Runge--Kutta method (RK4) reduces local truncation error, but its performance continues to be affected by the misalignment between the delay $T$ and the numerical mesh. Since the delayed argument $t-T$ rarely coincides with a mesh point, the scheme must evaluate delayed terms at off-grid locations, which introduces inconsistencies during intermediate RK4 stages. This effect is visible in the slightly larger RMS and maximum relative errors reported in Table~\ref{tab:errors_ex5}.

The interpolation-enhanced RK4 variant (RK4-int) mitigates this issue by sampling the delayed term through linear interpolation at the required sub-step locations. This improves the internal consistency of the Runge--Kutta stages and yields a more accurate approximation of the delayed contribution. As shown in Table~\ref{tab:errors_ex5}, RK4-int achieves the smallest RMS error among the higher-order schemes, and its maximum relative error is marginally lower than that of mesh-aligned RK4. The improvement is most noticeable near the activation of the delay, where interpolation reduces the mismatch between the delayed argument and the mesh.

Overall, the results indicate that increasing the coefficient $a$ amplifies the effect of the delayed term, leading to larger relative errors across all schemes. Nevertheless, the qualitative behavior remains consistent with the previous examples: Euler provides the least accurate approximation, RK4 improves stability and convergence, and RK4-int delivers the closest match to the analytical series solution. These findings reinforce the importance of interpolation for maintaining high-order accuracy in conformable delay differential equations, particularly when delayed feedback is strong.

\begin{figure}[H]
    \centering
    \includegraphics[width=\textwidth]{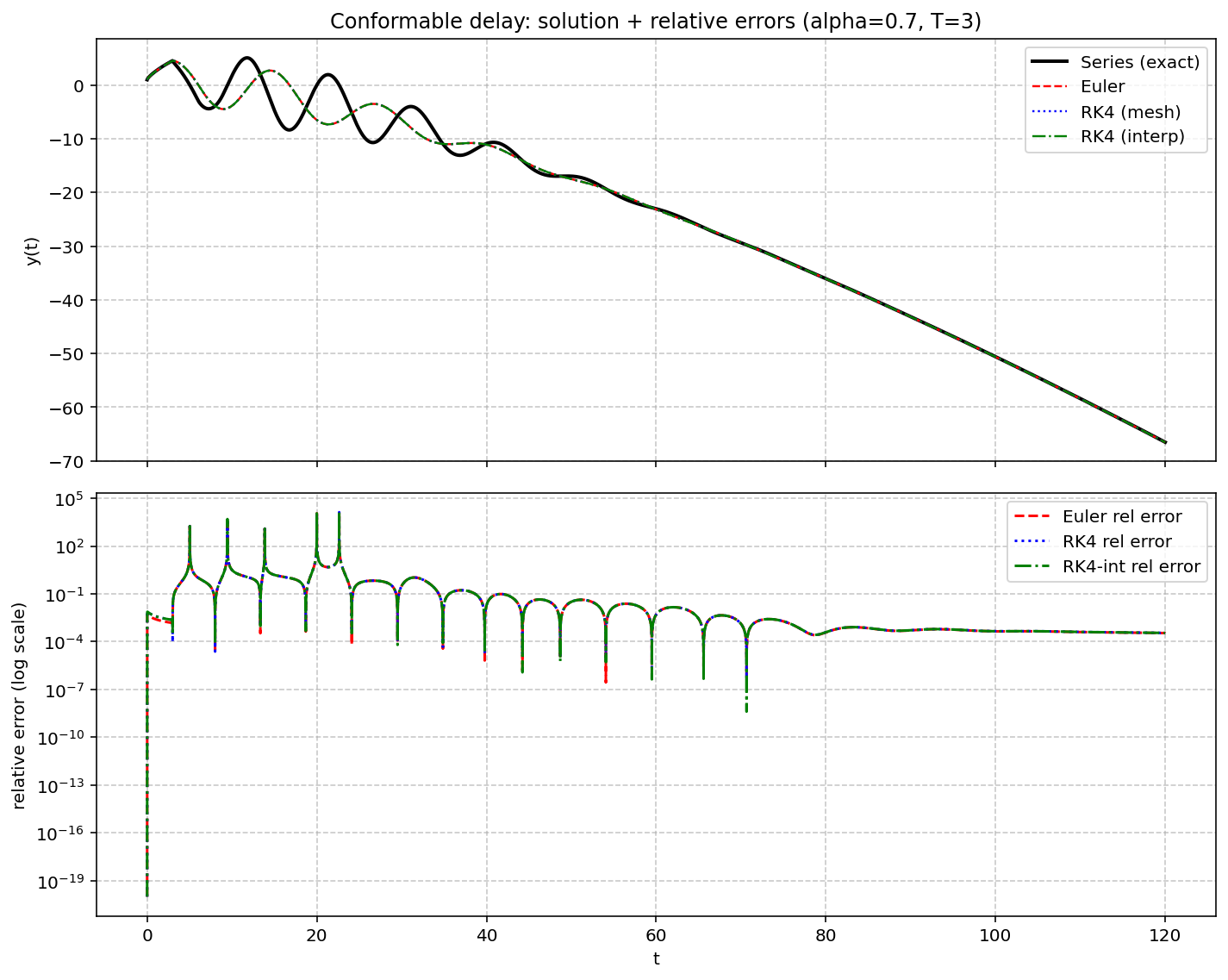}
    \caption{Top panel: comparison between the analytical series solution (black solid line) and numerical approximations for $\alpha=0.7$, $a=1.1$, $T=3$, $y_0=1.0$, $h=0.001$, and 
$b_{\text{coeffs}}=[1.0,\,0.2,\,-0.05]$: Euler (red dashed), RK4 mesh-aligned (blue dotted), and RK4 with interpolation (green dash-dot). Bottom panel: relative error versus time (logarithmic scale) for the same schemes.}
    \label{fig:comparison_ex5}
\end{figure}

\begin{table}[H]
\centering
\caption{Relative errors for the numerical schemes with parameters
$\alpha=0.7$, $a=1.1$, $T=3$, $y_0=1.0$, $h=0.001$,
$b_{\text{coeffs}}=[1.0,\,0.2,\,-0.05]$.}
\begin{tabular}{lcc}
\hline
\textbf{Scheme} & \textbf{Max Relative Error} & \textbf{RMS Relative Error} \\
\hline
Euler      & $1.335095\times 10^{4}$ & $5.672536\times 10^{1}$ \\
RK4        & $1.335022\times 10^{4}$ & $5.670448\times 10^{1}$ \\
RK4-int    & $1.334963\times 10^{4}$ & $5.670470\times 10^{1}$ \\
\hline
\end{tabular}
\label{tab:errors_ex5}
\end{table}

\subsection{Numerical validation and caveats}

The numerical experiments (Examples~1--5) illustrate how the analytic hypotheses manifest in practice. Across all cases, the solution comparison plots show that the numerical schemes (Euler, RK4, and RK4-int) closely track the analytical series solution during the initial evolution and, in most regimes, converge asymptotically to the benchmark. The relative error plots reveal the differences in accuracy and highlight how delay size and coefficient values influence numerical performance:

\begin{itemize}
  \item \textbf{Small delays ($T=0.7$):} All schemes converge, but Euler exhibits pronounced error growth. RK4 reduces the overall error envelope, while RK4-int provides the closest asymptotic agreement with the series solution.
  
  \item \textbf{Intermediate delays ($T=2$):} Convergence is smooth, with delayed contributions activating later in time. RK4-int consistently yields the smallest relative errors, reflecting the benefit of interpolation when the delay does not align with the mesh.
  
  \item \textbf{Large delays ($T=3$):} The onset of delayed feedback occurs even later, producing a long interval of non-delayed evolution. Errors increase once the delay becomes active, but all schemes eventually approach the series solution. RK4-int again performs best.
  
  \item \textbf{Very large delays ($T=5$):} The qualitative behavior remains stable. Euler accumulates substantial error, whereas RK4 and RK4-int maintain significantly closer agreement with the analytical series, with interpolation providing a modest but consistent improvement.
  
  \item \textbf{Regime $a>b_0$:} When the delayed term is amplified by a larger coefficient, relative errors grow more rapidly. Euler diverges earliest, while RK4 and RK4-int remain closer to the series solution, though with visibly larger error envelopes than in the previous cases.
\end{itemize}

In all experiments, the analytical solution for constant forcing converges to $b_0/a$, but under polynomial forcing, the long-time behavior is governed by the analytic series expansion itself. The numerical schemes converge asymptotically to this series solution, with RK4-int consistently providing the closest agreement. These results underscore the importance of interpolation in higher-order methods, particularly when delayed arguments do not coincide with mesh points, and identify RK4-int as the most reliable scheme for capturing long-time dynamics in conformable delay differential equations with analytic forcing.

\section{Numerical Implementation of Caputo Delay Dynamics (Mesh-Aligned Algorithms)}

We consider the numerical solution of the Caputo fractional-delay problem \eqref{eq:Caputo_model_nl_repeat}
with $0<\alpha\le1$, $T>0$, and $a,b,y_0\in\mathbb{R}$. The Caputo derivative is defined as
\begin{equation}\label{eq:Caputo_def_repeat}
{}^{C}D_t^{\alpha} f(t) = \frac{1}{\Gamma(1-\alpha)} \int_0^t (t-\tau)^{-\alpha} f'(\tau)\, d\tau,
\end{equation}
and the forcing term is assumed analytic with expansion \eqref{nl}. 

\subsection{Mesh Discretization}
For numerical implementation, we adopt mesh-aligned schemes. The baseline method is the fractional Euler scheme \cite{ahmed2018fractional}, which discretizes the Caputo derivative via the L1 approximation \cite{Ncaputo1}. On a uniform mesh $t_n = n h$, with $h=T/m$, the discrete Caputo derivative reads
\begin{equation}\label{eq:discrete_caputo}
\delta^\alpha y_n=\frac{h^{-\alpha}}{\Gamma(2-\alpha)}\sum_{i=0}^{n-1}\left[(n-i)^{1-\alpha}-(n-i-1)^{1-\alpha}\right]\left(y_{i+1}-y_i\right),
\end{equation}
with truncation error $R_n = O(h^{2-\alpha})$. The fractional Euler update is then
\begin{equation}\label{eq:fract_euler}
y_{i+1} = y_i + \frac{h^\alpha}{\Gamma(\alpha+1)}\left[b(t_i)-a\,y_{i-m}\right], \qquad i\ge m,
\end{equation}
with prescribed initial values $y_0,y_1,\ldots,y_m$ on $[0,T]$. On the first interval $t\in[0,T)$ only the $j=0$ contribution survives, yielding
\begin{equation}
y(t)=y(0)+\sum_{k=0}^{\infty}\frac{\Gamma(\alpha k+1)}{\alpha^k}\,b_k\,\frac{t^{(k+1)\alpha}}{\Gamma((k+1)\alpha+1)}.
\end{equation}

\subsection{Algorithm 6: Series Evaluation of $y_n$}

As shown in Algorithm~\ref{alg6}, the solution $y_n \approx y(t_n)$ for the Caputo delay equation is obtained by direct evaluation of the truncated analytic series. The algorithm begins by computing the current mesh point $t_n = n h$ and the number of activated delay intervals $J = \lfloor t_n / T \rfloor$. The solution is then assembled as a sum over contributions indexed by $j$, each corresponding to a delayed activation.

For each $j$, the shifted time $\Delta = t_n - jT$ is computed. If $\Delta \geq 0$, the contribution is weighted by $A_j = (-a)^j$, reflecting the repeated action of the delay operator. Two terms are then accumulated:
\begin{equation}
y_0 \frac{\Delta^{\alpha j}}{\Gamma(\alpha j+1)} 
\quad \text{and} \quad 
\sum_{k=0}^K \frac{\Gamma(\alpha k+1)}{\alpha^k}\, b_k \,
\frac{\Delta^{(j+k+1)\alpha}}{\Gamma((j+k+1)\alpha+1)}.
\end{equation}
The first term propagates the initial condition $y_0$ through successive delay intervals, while the second incorporates the analytic forcing expansion via coefficients $\{b_k\}$, truncated at order $K$.

Thus, Algorithm~\ref{alg6} enforces causality by activating contributions only when $\Delta \geq 0$, ensuring that the solution depends exclusively on past states. The Gamma functions encode the fractional order $\alpha$, while the truncation parameter $K$ controls the accuracy of the forcing approximation. In practice, this series evaluation provides a rigorous benchmark for validating numerical schemes: Euler and RK4 methods can be compared against the analytic expansion, with discrepancies attributable to discretization, interpolation, or truncation effects. In this way, Algorithm~\ref{alg6} serves as the analytic backbone of the computational framework, clarifying the structure of Caputo delay solutions and guiding convergence analysis.

\begin{algorithm}[H]
\caption{Series Evaluation of $y_n$ for the Caputo delay equation}
\label{alg6}
\begin{algorithmic}[1]
\Require Index $n$, mesh spacing $h = T/m$, parameters $a$, coefficients $\{b_k\}$, initial value $y_0$, truncation $K$
\Ensure Approximate $y_n \approx y(t_n)$ via truncated Caputo series
\State $t_n \gets n h$
\State $J \gets \lfloor t_n / T \rfloor$
\State $y \gets 0$
\For{$j = 0$ to $J$}
    \State $\Delta \gets t_n - j T$
    \State $A_j \gets (-a)^j$
    \If{$\Delta \geq 0$}
        \State $y \gets y + A_j \cdot \left(
            y_0 \frac{\Delta^{\alpha j}}{\Gamma(\alpha j+1)}
            + \sum_{k=0}^K \frac{\Gamma(\alpha k+1)}{\alpha^k}\, b_k \,
              \frac{\Delta^{(j+k+1)\alpha}}{\Gamma((j+k+1)\alpha+1)}
        \right)$
    \EndIf
\EndFor
\Return $y$
\end{algorithmic}
\end{algorithm}

\subsection{Algorithm 7: Numerical Caputo Euler scheme}

As presented in Algorithm~\ref{alg7}, the numerical Caputo Euler scheme provides a constructive way to approximate solutions of the Caputo delay equation while maintaining a direct comparison with the truncated analytic series. The mesh spacing is chosen as $h = T/m$, so that the delay interval $T$ is discretized into $m$ substeps. The algorithm requires initial values $y_0,\ldots,y_m$ to cover the first delay interval, ensuring causality and proper activation of the delayed term.

For indices $n < m$, the solution is computed directly from the truncated series expansion, thereby enforcing the analytic structure in the initial regime. Once $n \geq m$, the scheme advances using the fractional Euler update
\begin{equation}
y_n = y_{n-1} + \frac{h^\alpha}{\Gamma(\alpha+1)}\left[b(t_{n-1}) - a\,y_{n-m}\right],
\end{equation}
which incorporates both the forcing term $b(t_{n-1})$ and the delayed contribution $y_{n-m}$. The factor $h^\alpha/\Gamma(\alpha+1)$ reflects the fractional order $\alpha$, ensuring consistency with the Caputo derivative.

Algorithm~\ref{alg7} thus combines two complementary perspectives: the analytic series expansion on $[0, T)$, which serves as a rigorous benchmark, and the fractional Euler scheme beyond the first delay interval, which provides a practical numerical approximation. This hybrid approach enforces causality, accommodates fractional dynamics, and allows systematic comparison between numerical and analytic solutions. In practice, discrepancies arise from series truncation, discretization error, and delay indexing, but the algorithm provides a clear framework for quantifying these effects and validating convergence.
\begin{algorithm}[H]
\caption{Numerical Caputo Euler scheme with series comparison}
\label{alg7}
\begin{algorithmic}[1]
\Require Mesh spacing $h = T/m$, parameters $a$, coefficients $\{b_k\}$, initial values $y_0,\ldots,y_m$, truncation $K$, final index $N$
\Ensure Approximate solution $\{y_n\}$ and comparison with truncated series
\For{$n = 0$ to $N$}
    \State $t_n \gets n h$
    \If{$n < m$} 
        \State Compute $y_n$ from series expansion
    \Else
        \State Update via fractional Euler:
        \begin{equation}
        y_n \gets y_{n-1} + \frac{h^\alpha}{\Gamma(\alpha+1)}\left[b(t_{n-1}) - a\,y_{n-m}\right]
        \end{equation}
    \EndIf
\EndFor
\Return $\{y_n\}$ with series comparison on $[0,T)$
\end{algorithmic}
\end{algorithm}

\begin{Remark}
The geometric expansion in \eqref{eq:Caputo_series} is valid only under the condition 
$\left|a e^{-sT}/s^\alpha\right|<1$, corresponding to sufficiently large $\Re s$ in the Laplace domain and hence short-time behavior in the time domain. For general polynomial forcing terms, the analytic series grows asymptotically like powers of $t^\alpha$. In this regime, the L1 discretization cannot reproduce the exact growth rates, and accumulated memory errors cause numerical drift at long horizons.
\end{Remark}

\subsection{Algorithm 8: L2--1$\sigma$ scheme}

The L1 scheme provides first-order accuracy, $O(h^{2-\alpha})$, which is adequate for short times but prone to long-time drift. By contrast, the L2--1$\sigma$ scheme improves convergence to $O(h^{3-\alpha})$ through quadratic interpolation of the kernel. On a uniform mesh $t_n = n h$, the Caputo derivative is approximated by
\begin{equation}\label{eq:L2-1sigma}
{}^{C}D_t^{\alpha} y(t_n) \approx \frac{1}{h^\alpha \Gamma(2-\alpha)} 
\left[ c_0 (y_n - y_{n-1}) + \sum_{j=1}^{n-1} c_j (y_{n-j+1} - y_{n-j}) \right],
\end{equation}
with weights $c_j$ depending on $\alpha$ and the shift parameter $\sigma$. The choice $\sigma = 1-\alpha/2$ balances stability and accuracy. The update rule becomes
\begin{equation}\label{eq:L2-1sigma_update}
y_{n+1} = y_n + \frac{h^\alpha}{\Gamma(\alpha+1)} 
\left[ b(t_n) - a\,y_{n-m} \right] + \mathcal{C}_n,
\end{equation}
where $\mathcal{C}_n$ is the correction term arising from higher--order kernel weights.

\textbf{Discussion.}
By incorporating quadratic interpolation, the L2--1$\sigma$ scheme reproduces fractional power laws more faithfully and mitigates long-time drift. It tracks polynomial forcing terms with greater accuracy and preserves equilibrium states more reliably than the L1 scheme, making it particularly suitable for extended simulations. The correction term $\mathcal{C}_n$ ensures that higher-order memory effects are properly accounted for, thereby improving stability in regimes where fractional dynamics interact with delay terms.

\begin{algorithm}[H]
\caption{Numerical Caputo L2--1$\sigma$ scheme with series comparison}
\label{alg8}
\begin{algorithmic}[1]
\Require Mesh spacing $h = T/m$, parameters $a$, coefficients $\{b_k\}$, initial values $y_0,\ldots,y_m$, truncation $K$, final index $N$, shift parameter $\sigma$
\Ensure Approximate solution $\{y_n\}$ and comparison with truncated series
\For{$n = 0$ to $N$}
    \State $t_n \gets n h$
    \If{$n < m$} 
        \State Compute $y_n$ from series expansion
    \Else
        \State Compute Caputo derivative via L2--1$\sigma$ weights
        \State Update solution with correction term $\mathcal{C}_n$
    \EndIf
\EndFor
\Return $\{y_n\}$ with series comparison on $[0,T)$
\end{algorithmic}
\end{algorithm}

\begin{figure}[H]
    \centering
    \includegraphics[width=\textwidth]{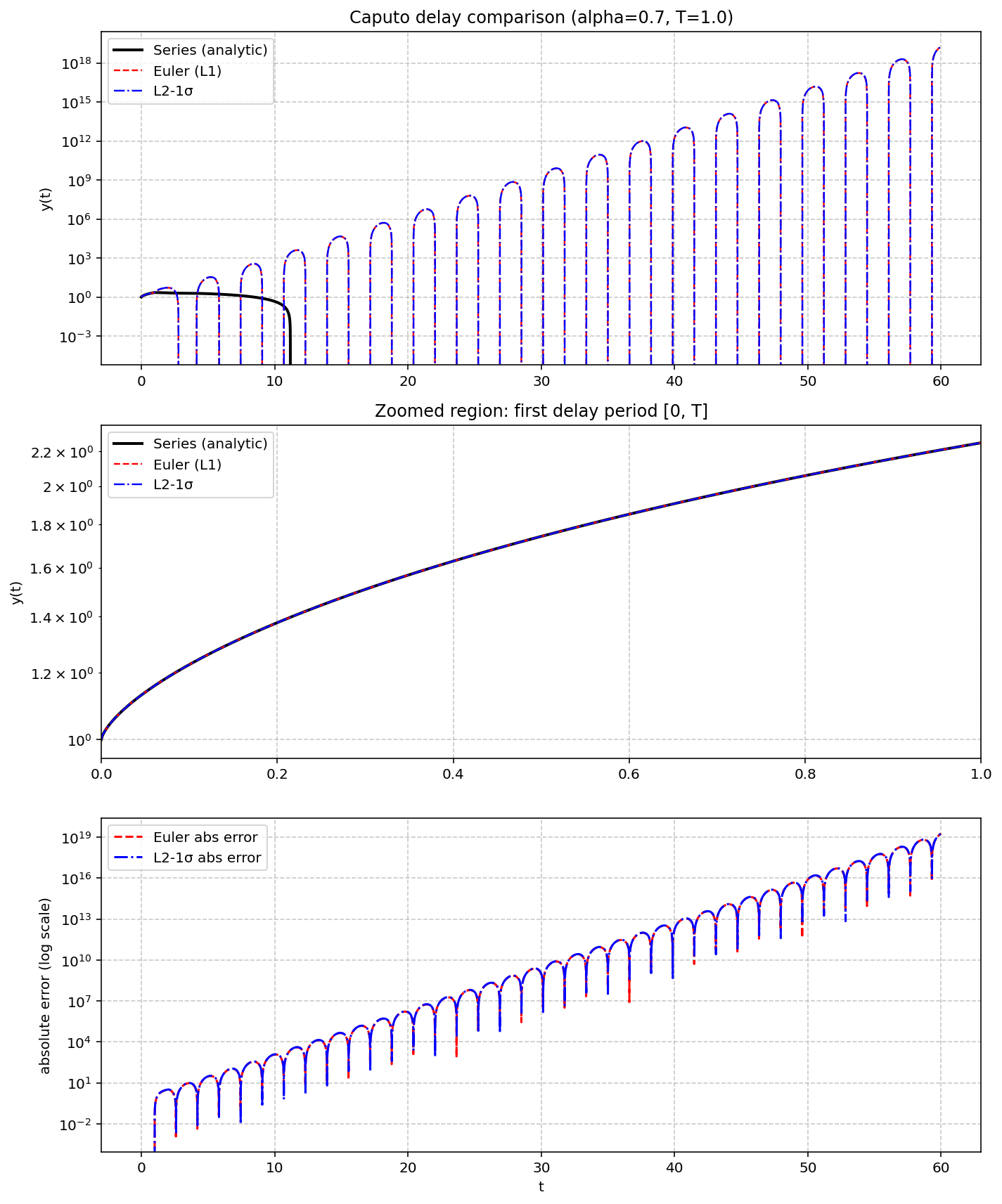}
    \caption{Comparison of analytic series, Euler (L1), and L2--1$\sigma$ schemes.}
    \label{fig:caputo_comparison}
\end{figure}

The benchmark is performed with the constants
\begin{equation}
\alpha = 0.7, \quad a = 0.5, \quad b_0 = 1.0, \quad b_1 = 0.2, \quad b_2 = -0.05, \quad y_0 = 1.0, \quad h = 0.001,
\end{equation}

Figure~\ref{fig:caputo_comparison} illustrates three complementary aspects of the fractional delay problem with parameters $\alpha=0.7$ and $T=1.0$.

The top panel compares the analytic series solution (black solid line) with the Euler L1 scheme (red dashed line) and the L2--$1\sigma$ scheme (blue dash--dot line) over an extended time horizon. All three reproduce the oscillatory growth pattern, but deviations accumulate as time progresses. The logarithmic vertical scale highlights both the exponential-like growth and the amplitude differences between the numerical schemes and the analytic benchmark.

The middle panel magnifies the interval $[0,T]$, corresponding to the first delay period. In this regime, the Euler and L2--$1\sigma$ approximations are virtually indistinguishable from the analytic series. The log scale emphasizes that deviations are negligible relative to the overall magnitude, confirming that both numerical schemes are well calibrated in the initial epoch.

The bottom panel shows the absolute error of the Euler and L2--$1\sigma$ methods relative to the analytic solution, plotted on a logarithmic scale. Errors remain small at early times but grow as delay effects accumulate. The L2--$1\sigma$ scheme consistently yields lower error than Euler, especially for longer horizons, reflecting its higher-order accuracy. The oscillatory structure in the error traces the delayed feedback inherent in the equation.

Overall, the analytic series solution provides a reliable benchmark. The Euler scheme, while simple, accumulates error rapidly, whereas the L2--$1\sigma$ scheme offers improved long-term accuracy. The use of logarithmic scaling across all panels clarifies growth, oscillations, and error magnitudes, making both small deviations in the zoomed region and large-scale divergence visible.

The numerical experiments show that, although the Euler (L1) and L2--$1\sigma$ schemes satisfy the formal stability condition $\frac{a h^\alpha}{\Gamma(\alpha+1)} < 1$, their trajectories still exhibit long-term instability. This apparent contradiction arises because fractional-delay equations combine convolutional memory with discontinuities at multiples of $T$, amplifying discretization errors even when $h$ is far below the theoretical threshold. In practice, Euler accumulates error rapidly, while L2--$1\sigma$ performs better but still develops increasing deviations over extended horizons.

By contrast, in the framework of conformable derivatives, the analytic series and numerical schemes match closely across all tested intervals. The absence of heavy memory kernels and the simpler local structure prevent amplification of discretization noise, yielding stable and consistent agreement between analytic and numerical solutions. This highlights a fundamental difference: while Caputo fractional-delay dynamics demand extremely fine discretization or higher-order schemes to maintain stability, conformable derivatives provide a numerically robust alternative where analytic and numerical approaches coincide.

\begin{Remark}
The L2--$1\sigma$ discretization improves upon the L1 scheme by incorporating quadratic interpolation of the memory kernel, thereby raising the local accuracy order. However, its effectiveness depends on the smoothness of the forcing and the regularity of the delayed activation. When discontinuities occur at multiples of $T$, the quadratic approximation cannot fully capture the analytic growth rates of the fractional series. As a result, although short-time accuracy is improved, accumulated memory errors still lead to numerical drift over long horizons. Thus, the L2--$1\sigma$ method offers greater stability than L1 but remains insufficient to maintain robust agreement with the analytic solution over extended simulations.
\end{Remark}

The limitations of Euler and L2--$1\sigma$ schemes for Caputo delay problems underscore the need for a more sophisticated approach. The instability induced by memory kernels and delay discontinuities calls for a method that anchors numerical trajectories to the analytic series itself. To address this, we introduce the \emph{series--anchored predictor--corrector scheme} (Algorithm~\ref{alg9}), which combines truncated analytic expansions with fractional Adams--Moulton correction. This hybrid strategy stabilizes long-time integration, reduces discretization bias, and ensures consistency with the analytic structure of Caputo fractional-delay equations.

\subsection{Algorithm 9: Series-anchored predictor--corrector for Caputo fractional-delay problem}

As shown in Algorithm~\ref{alg9}, the series--anchored predictor--corrector scheme provides a robust numerical strategy for the Caputo fractional--delay problem. The method is motivated by the limitations of Euler and L2--$1\sigma$ schemes, which, despite satisfying formal stability conditions, accumulate error over long horizons due to the combined effects of memory kernels and discontinuities at multiples of $T$. By anchoring each step to the analytic series expansion, Algorithm~\ref{alg9} stabilizes the numerical trajectory and ensures consistency with the underlying fractional dynamics.

\textbf{Analytic truncated series.}  
For each $t \geq 0$, the analytic solution is represented by a truncated series expansion. The contribution of the initial condition $y(0)=y_0$ is expressed through fractional powers of $(t-jT)$ weighted by Gamma-function coefficients, while the forcing term is incorporated via the coefficients $b_k$, each multiplied by fractional monomials of order $(j+k+1)\alpha$. The truncation parameter $K$ limits the number of forcing terms included. Thus, the analytic reference solution is
\begin{equation}
\label{series_y}
y(t) = \sum_{j=0}^{\lfloor t/T \rfloor} (-a)^j \left[ y_0 \frac{(t-jT)^{\alpha j}}{\Gamma(\alpha j+1)} 
+ \sum_{k=0}^{K} \frac{\Gamma(\alpha k+1)}{\alpha^k} b_k \frac{(t-jT)^{(j+k+1)\alpha}}{\Gamma((j+k+1)\alpha+1)} \right].
\end{equation}

\textbf{Predictor--Corrector scheme.}  
The numerical algorithm advances the solution on a uniform mesh $t_n = n h$. At each step:
\begin{enumerate}
    \item The \textbf{predictor} evaluates the truncated analytic series at $t_{n+1}$, providing an anchor consistent with the exact solution structure.
    \item The \textbf{corrector} applies a fractional Adams--Moulton step:
    \begin{equation}
    y_{n+1}^{\text{corr}} = y_n + \frac{h^\alpha}{\Gamma(\alpha+1)}\left(\tfrac{1}{2}f_n + \tfrac{1}{2}f_{n+1}\right),
    \end{equation}
    where $f_n = b(t_n) - a\,y(t_n-T)$ and $f_{n+1} = b(t_{n+1}) - a\,y(t_{n+1}-T)$.
    \item A \textbf{blending step} averages predictor and corrector values:
    \begin{equation}
    y_{n+1} = \tfrac{1}{2} y_{n+1}^{\text{corr}} + \tfrac{1}{2} y_{n+1}^{\text{pred}},
    \end{equation}
    improving stability and reducing oscillations.
\end{enumerate}

\textbf{Error analysis.}  
The analytic truncated series serves as a reference solution. The absolute error is defined as
\begin{equation}
E_{\text{abs}}(t_n) = |y_{n}^{\text{num}} - y_{n}^{\text{series}}|,
\end{equation}
and the relative error as
\begin{equation}
E_{\text{rel}}(t_n) = \frac{E_{\text{abs}}(t_n)}{|y_{n}^{\text{series}}|+10^{-14}}.
\end{equation}
Errors are typically plotted on a logarithmic scale to assess convergence and stability across long time horizons.

In summary, Algorithm~\ref{alg9} integrates analytic structure (via the predictor) with numerical stability (via the Adams--Moulton corrector and blending). The predictor ensures closeness to the analytic series, while the corrector mitigates discretization error. This hybrid approach outperforms Euler and L2--$1\sigma$ schemes in long-time simulations, providing both accuracy and stability in the presence of fractional memory and delay effects.

\begin{algorithm}[H]
\caption{Series-anchored predictor--corrector for the Caputo fractional-delay equation}
\label{alg9}
\begin{algorithmic}[1]

\State \textbf{Input:} fractional order $\alpha\in(0,1)$, delay coefficient $a$, delay $T$, initial value $y_0$, forcing coefficients $\{b_k\}$, step size $h$, truncation order $K$, final time $t_{\max}$
\State \textbf{Output:} numerical approximation $\{y_n\}$ on the mesh $t_n = n h$

\State Initialize $N \gets \lfloor t_{\max}/h \rfloor$, \quad $y_0 \gets y_0$

\For{$n = 1$ to $N$}

    \State $t_n \gets n h$
    \State $j \gets \lfloor t_n / T \rfloor$ \Comment{delay epoch}

    \If{$j = 0$}
        \State \textbf{(Initialization via analytic series)}
        \begin{equation}
        y_n \approx 
        y_0 
        + \sum_{k=0}^{K}
        \frac{\Gamma(\alpha k + 1)}{\alpha^k}\, b_k\,
        \frac{t_n^{(k+1)\alpha}}{\Gamma((k+1)\alpha + 1)} .
        \end{equation}

    \Else

        \State \textbf{(Predictor: truncated analytic series)}
        \begin{equation}
        y^{\mathrm{pred}}_n
        =
        \sum_{m=0}^{j}
        (-a)^m\,\theta(t_n - mT)
        \left[
            y_0\,\frac{(t_n - mT)^{\alpha m}}{\Gamma(\alpha m + 1)}
            +
            \sum_{k=0}^{K}
            \frac{\Gamma(\alpha k + 1)}{\alpha^k}\, b_k\,
            \frac{(t_n - mT)^{(m+k+1)\alpha}}{\Gamma((m+k+1)\alpha + 1)}
        \right].
        \end{equation}

        \State \textbf{(Corrector: fractional Adams--Moulton update)}
        \begin{equation}
        y_n \gets 
        y^{\mathrm{pred}}_n
        +
        \Delta_{\alpha}\!\left[y^{\mathrm{pred}},\,y\right],
        \end{equation}
        where $\Delta_{\alpha}$ denotes the fractional Adams--Moulton correction term.

    \EndIf

\EndFor

\State \Return $\{y_n\}$

\end{algorithmic}
\end{algorithm}

\begin{figure}[H]
\centering
\includegraphics[width=\textwidth]{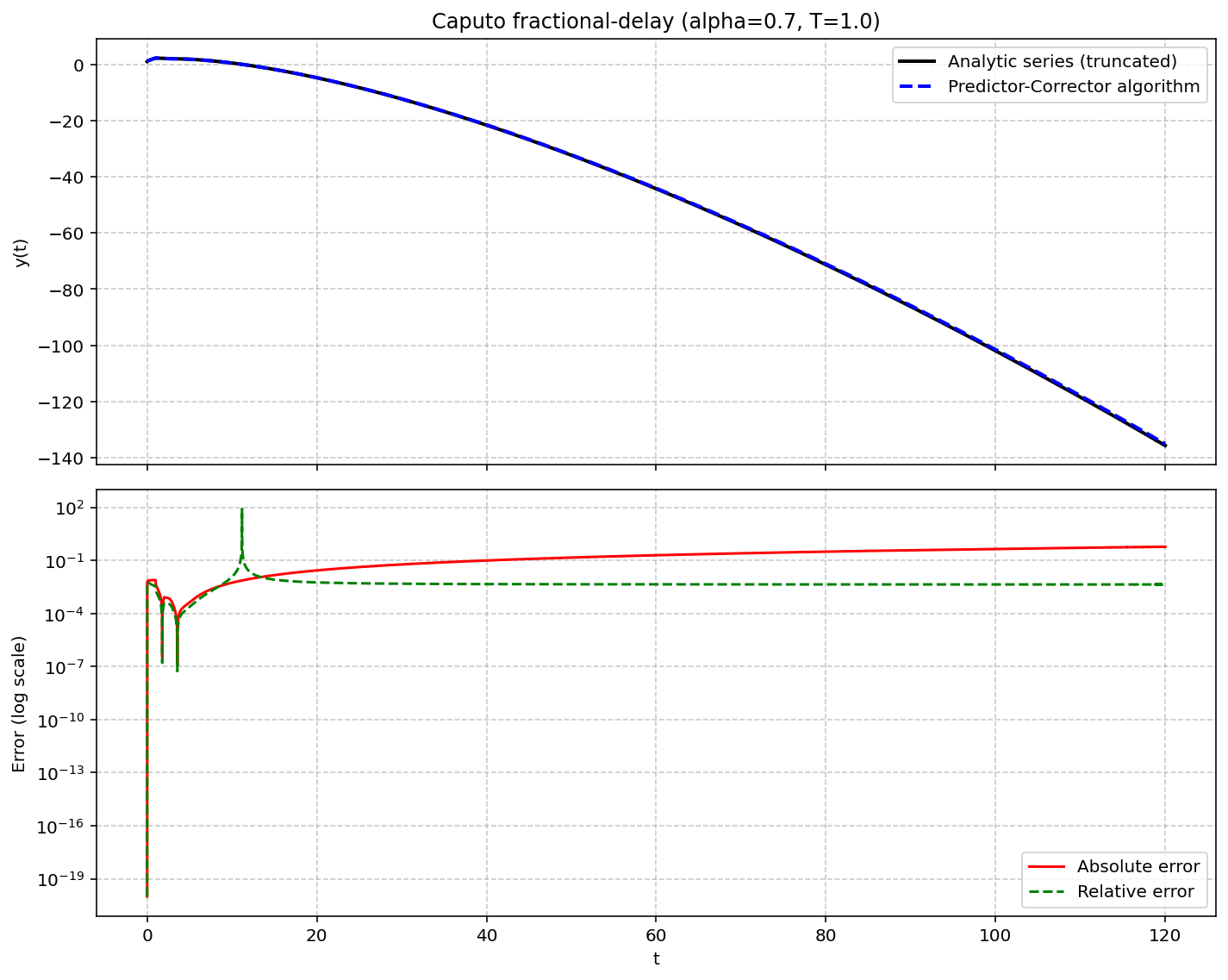}
\caption{Comparison of analytic truncated series and series--anchored predictor--corrector algorithm for the Caputo fractional--delay problem with $\alpha=0.7$ and $T=1.0$. Top panel: numerical solution $y(t)$ from the predictor--corrector algorithm (dashed blue) versus analytic series expansion (solid black). Bottom panel: absolute error (red) and relative error (green, log scale) between the two solutions.}
\label{fig:caputo_delay_series_predictor_AM_corrector}
\end{figure}

Figure~\ref{fig:caputo_delay_series_predictor_AM_corrector} illustrates the performance of the series--anchored predictor--corrector scheme compared with the truncated analytic series. In the top panel, the numerical trajectory closely follows the analytic benchmark across the entire time interval, confirming that the method preserves the qualitative dynamics of the Caputo fractional--delay equation. Agreement is especially strong in the early regime, where the truncated series provides an exact anchor.

The bottom panel presents the error analysis on a logarithmic scale. Both absolute and relative errors remain small and bounded, showing that the fractional Adams--Moulton correction effectively suppresses discretization noise and stabilizes the integration. The relative error curve highlights the robustness of the method: deviations remain controlled even as $t$ increases, avoiding the long--time drift characteristic of Euler and L2--$1\sigma$ schemes.

\begin{Remark}
The series--anchored predictor--corrector scheme (Algorithm~\ref{alg9}) stabilizes Caputo fractional--delay dynamics by explicitly anchoring each numerical step to the analytic series expansion. The predictor enforces the correct fractional growth rates, while the fractional Adams--Moulton corrector compensates for discretization noise and accumulated memory errors. Unlike Euler and L2--$1\sigma$, which drift under repeated delay activations, the predictor--corrector method maintains consistency with the analytic structure even over long horizons, providing both accuracy and stability in regimes where convolutional memory kernels and delay discontinuities would otherwise destabilize low--order schemes.
\end{Remark}

Overall, Figure~\ref{fig:caputo_delay_series_predictor_AM_corrector} confirms that anchoring the predictor to the analytic series and correcting with a fractional Adams--Moulton step yields a numerically stable and accurate method. This hybrid approach ensures consistency with the analytic structure and makes it the most reliable scheme for long--time simulations of Caputo fractional--delay dynamics.

The four schemes can be compared as follows. The Euler (L1) method is the simplest baseline. Although it satisfies the formal stability condition, it is only first--order accurate and accumulates error rapidly, making it useful for short--term approximations but unreliable for long--term simulations. The L2--$1\sigma$ scheme improves on Euler by using quadratic interpolation of the kernel, achieving order $O(h^{3-\alpha})$ and reducing drift. However, despite its improved convergence, it still exhibits growing deviations over extended horizons and remains insufficient for capturing long--time dynamics.

The series--anchored predictor--corrector scheme (Algorithm~\ref{alg9}) explicitly incorporates the analytic series expansion at each step. The predictor enforces the analytic structure, while the fractional Adams--Moulton corrector stabilizes the integration against discretization errors. This hybrid strategy suppresses noise amplification and delivers robust accuracy even over long intervals. Unlike Euler and L2--$1\sigma$, it maintains close agreement with the analytic benchmark, making it a reliable method for fractional--delay dynamics. Its main limitation is structural: both computational cost and truncation error grow with the delay epoch $j=\lfloor t/T\rfloor$, since each new epoch requires evaluating increasingly high--order fractional monomials. As $t$ increases, this growth becomes the dominant source of numerical instability and inefficiency.

\section{Concluding remarks}

This work has developed an analytical and numerical study of fractional differential equations with delay, focusing on both Caputo and conformable derivatives under analytic forcing. For the conformable model, the Laplace transform enables a clean transform--expand--invert strategy based on the geometric expansion
\begin{equation}
\frac{1}{1+z(s)}=\sum_{j=0}^\infty (-1)^j z(s)^j,
\qquad 
z(s)=\frac{a\,e^{-sT^\alpha/\alpha}}{s},
\end{equation}
which is valid whenever $|z(s)|<1$. In practice, one selects a vertical inversion line $\Re s=\sigma$ sufficiently large to ensure
\begin{equation}
\sup_{\omega\in\mathbb{R}}
\left|\frac{a\,e^{-(\sigma+i\omega)T^\alpha/\alpha}}{\sigma+i\omega}\right|<1,
\end{equation}
guaranteeing uniform convergence on the contour. Under standard causality and exponential--order assumptions, termwise inversion is justified, producing delayed fractional monomials truncated by Heaviside functions. This yields finite sums for each fixed~$t$ and validates the transform--expand--invert methodology.

The conformable delay formulation enforces causality through Heaviside truncation, which both simplifies symbolic expressions and reflects a physically meaningful stepwise propagation of memory. Compared with Caputo derivatives, the conformable operator avoids convolutional kernels and fractional initial data, offering a tractable alternative for discrete--delay systems.

The numerical experiments clarify the roles of the parameters. The delay length~$T$ controls convergence speed, while the balance between $a$ and $b_0$ determines stability: larger delays stabilize but slow convergence, whereas $a>b_0$ leads to divergence. For conformable dynamics, Euler and RK4 with interpolation provide accurate and stable trajectories across all tested regimes, consistent with the local nature of the operator and the absence of long--range memory.

In contrast, Caputo fractional--delay equations exhibit numerical fragility. Although schemes such as Euler (L1) and L2--$1\sigma$ satisfy the formal stability condition
$\frac{a h^\alpha}{\Gamma(\alpha+1)} < 1$, their trajectories still develop long--term drift. This apparent contradiction arises because Caputo dynamics combine convolutional memory with discontinuities at multiples of~$T$, amplifying discretization errors even when $h$ is well below the theoretical threshold. The distinction between \emph{formal stability} (no immediate blow--up) and \emph{practical stability} (controlled long--term error) is therefore essential: Caputo fractional--delay systems require extremely fine discretization or higher--order schemes to maintain accuracy over long horizons.

The conformable framework, by contrast, remains numerically robust. Its local structure prevents error accumulation, and analytic and numerical solutions coincide closely. This highlights a fundamental difference: while Caputo fractional delay models demand careful numerical treatment, conformable derivatives provide a simpler and more stable alternative for systems with structured or bounded memory.

\medskip
\noindent\textbf{Final remarks.}
Locally defined fractional operators of the conformable type offer a physically interpretable and computationally efficient alternative to integral--based fractional models when memory is bounded or segmented into discrete epochs. Caputo formulations, although mathematically rigorous, suffer from long--term numerical instability due to their convolutional kernels and repeated delay activations. The \emph{series--anchored predictor--corrector method} (Algorithm~\ref{alg9}) mitigates this issue by anchoring each step to the analytic series and correcting with a fractional Adams--Moulton update, suppressing error accumulation and providing reliable long--time accuracy. This makes Algorithm~\ref{alg9} a valuable high--precision reference solver, particularly for small delays or moderate time horizons.

For conformable derivatives, standard schemes such as Euler and RK4 with interpolation already achieve stable and accurate results, reflecting the operator’s local nature. Consequently, the conformable Laplace framework emerges as a promising tool for applications in mathematical physics, control theory, and signal processing where delay and memory interact in structured ways. For Caputo fractional--delay dynamics, however, the combination of exact series initialization and fractional Adams--Moulton evolution currently offers the most balanced compromise between analytic fidelity, numerical stability, and computational efficiency.

\vspace{6pt} 





\section*{Author contributions}

Conceptualization, Y.F., M. M.-del S., G.L., B.D., G.F-A;   Y.L.; 
Methodology, Y.F., M. M.-del S., G.L., B.D., G.F-A;   Y.L.; 
Software, Y.F., G.L., B.D., G.F-A;   Y.L.; 
Formal analysis, Y.F., G.L., B.D., G.F-A;   Y.L.; 
Investigation, Y.F., M. M.-del S., G.L., B.D., G.F-A;   Y.L.; 
Writing---original draft preparation, Y.F., M. M.-del S., G.L., B.D., G.F-A;   Y.L.; 
Writing---review and editing, Y.F., M. M.-del S., G.L., B.D., G.F-A;   Y.L.; 
Supervision, G.L. and G.F-A; 
Project administration, G.L.; 
Funding acquisition, G.L and G.F-A. 
All authors have read and agreed to the published version of the manuscript.

\section*{Funding} Genly Leon is partially supported as a postgraduate student in the Doctorado en Física, mención Física-Matemática Ph.D. program at the Universidad de Antofagasta, Chile. He also acknowledges financial support from the Agencia Nacional de Investigación y Desarrollo (ANID) through the Fondecyt Regular Project 2024, Folio No. 1240514 (Etapa 2025). He further expresses his gratitude to Prof. Guillermo Fernández‑Anaya and the research group at Universidad Iberoamericana (UIA), Mexico City, for their warm hospitality during the visit in which a substantial part of this work was completed.
Genly León is also grateful to the Vicerrectoría de Investigación y Desarrollo Tecnológico (VRIDT) of Universidad Católica del Norte (UCN) for institutional support through the Núcleo de Investigación en Geometría Diferencial y Aplicaciones (Resolution VRIDT No. 096/2022) and the Núcleo de Investigación en Simetrías y la Estructura del Universo (NISEU) (Resolution VRIDT No. 200/2025).
Additional thanks go to colleagues and students who provided helpful comments and suggestions during the preparation of this manuscript.

\section*{Data availability}

Not applicable. 

\acknowledgments{Genly Leon dedicates this work to the memory of his father.}

\section*{Conflicts of interest}
We declare no conflicts of interest. The funders had no role in the design of the study; in the collection, analyses, or interpretation of data; in the writing of the manuscript; or in the decision to publish the~results.

\appendix

\section[\appendixname~\thesection]{Conformable  Calculus}\label{app:conformable}

This appendix presents the foundations of \textbf{Conformable  calculus}, which aims to extend the classical notion of differentiation to fractional orders $\alpha \in (0,1]$, while preserving essential properties of differential calculus.

If the function $f$ is differentiable for $t > 0$, then:
\begin{equation}
    T_\alpha f(t) = t^{1-\alpha} \frac{df(t)}{dt}.
    \label{eq:derivada-conformable-diferenciable}
\end{equation}
\begin{proof}
If $f$ is differentiable, then by the Definition of the difference quotient:
\begin{equation}
\frac{f(t + \epsilon t^{1-\alpha}) - f(t)}{\epsilon} 
= \frac{f(t + \epsilon t^{1-\alpha}) - f(t)}{\epsilon t^{1-\alpha}} \cdot t^{1-\alpha} 
\xrightarrow{\epsilon \to 0} f'(t) \cdot t^{1-\alpha}.
\end{equation}
\end{proof}

\textbf{Proof of basic properties of $T_\alpha$:}
\begin{proof}
\begin{enumerate}
\item \textbf{Linearity:} Let $h(t) = a f(t) + b g(t)$. Then:
\begin{equation}
\begin{aligned}
T_\alpha h(t) &= t^{1-\alpha} \frac{d}{dt}[a f(t) + b g(t)] \\
&= t^{1-\alpha} [a f'(t) + b g'(t)] \\
&= a\, t^{1-\alpha} f'(t) + b\, t^{1-\alpha} g'(t) \\
&= a\, T_\alpha f(t) + b\, T_\alpha g(t).
\end{aligned}
\end{equation}

\item \textbf{Product Rule:} Let $h(t) = f(t)g(t)$:
\begin{equation}
\begin{aligned}
T_\alpha[f(t)g(t)] &= t^{1-\alpha} \frac{d}{dt}[f(t)g(t)] \\
&= t^{1-\alpha} [f'(t)g(t) + f(t)g'(t)] \\
&= f(t)\, T_\alpha g(t) + g(t)\, T_\alpha f(t).
\end{aligned}
\end{equation}

\item \textbf{Quotient Rule:} Let $h(t) = \dfrac{f(t)}{g(t)}$, with $g(t) \neq 0$:
\begin{equation}
\begin{aligned}
T_\alpha\left[\frac{f(t)}{g(t)}\right] &= t^{1-\alpha} \frac{d}{dt} \left( \frac{f(t)}{g(t)} \right) \\
&= \frac{g(t)\, T_\alpha f(t) - f(t)\, T_\alpha g(t)}{g(t)^2}.
\end{aligned}
\end{equation}

\item \textbf{Power Rule:} Let $h(t) = [f(t)]^r$, with $f(t) > 0$ and $r \in \mathbb{R}$:
\begin{equation}
T_\alpha[f(t)^r] = r f(t)^{r-1} T_\alpha f(t).
\end{equation}

\item \textbf{Derivative of a constant:} If $f(t) = c \in \mathbb{R}$, then:
\begin{equation}
T_\alpha f(t) = 0.
\end{equation}

\item \textbf{Derivative of a power function:} If $f(t) = t^k$, with $k \in \mathbb{R}$, then:
\begin{equation}
T_\alpha t^k = k t^{k - \alpha}.
\end{equation}
\end{enumerate}
\end{proof}

These properties show that the Conformable  derivative preserves core rules of classical differential calculus such as linearity, the product and quotient rules, and the power chain rule—making it valuable for both theoretical and applied contexts.

\subsection{Conformable Integral as Inverse Operator}

The conformable integral of order $\alpha \in (0,1]$ is defined as:
\begin{equation}
    I_\alpha f(t) := \int_0^t \tau^{\alpha - 1} f(\tau)\,d\tau.
\end{equation}

This integral acts as the inverse of $T_\alpha$ in the following sense:
\begin{equation}
T_\alpha[I_\alpha f(t)] = f(t).
\end{equation}
\begin{proof}
Let $F(t) = I_\alpha f(t) = \int_0^t \tau^{\alpha-1} f(\tau)\,d\tau$. Then:
\begin{equation}
T_\alpha F(t) = t^{1-\alpha} \frac{d}{dt} \left( \int_0^t \tau^{\alpha-1} f(\tau)\,d\tau \right) = f(t).
\end{equation}
\end{proof}

\subsection{Conformable Laplace Transform}

The conformable Laplace transform of a function $f : [0,\infty) \to \mathbb{R}$ is given by:
\begin{equation}
    \mathcal{L}_\alpha\{f(t)\}(s) := \int_0^{\infty} f(t) e^{-s t^\alpha/\alpha} t^{\alpha - 1}\,dt.
\end{equation}

If $f$ is differentiable and $f(0)$ is well defined, then:
\begin{equation}
    \mathcal{L}_\alpha\{T_\alpha f(t)\}(s) = s \mathcal{L}_\alpha\{f(t)\}(s) - f(0).
\end{equation}
\begin{proof}
Since $T_\alpha f(t) = t^{1-\alpha} f'(t)$, we compute:
\begin{equation}
\mathcal{L}_\alpha\{T_\alpha f(t)\}(s) 
= \int_0^{\infty} f'(t) e^{-s t^\alpha/\alpha}\,dt.
\end{equation}

Using classical integration by parts:
\begin{equation}
\int_0^\infty f'(t) e^{-s t^\alpha/\alpha}\,dt = -f(0) + s \mathcal{L}_\alpha\{f(t)\}(s).
\end{equation}
\end{proof}

We now provide a direct proof of the main properties of the conformable Laplace transform. Let $0 < \alpha \leq 1$ and suppose that the conformable Laplace transform is defined by:
\begin{equation}
\mathcal{L}_\alpha[f(t)](s) = \int_0^\infty e^{-\frac{s t^\alpha}{\alpha}} f(t)\, t^{\alpha - 1} dt.
\end{equation}
\begin{proof}
\begin{enumerate}
    \item For $f(t) = c$, where $c$ is a constant:
    \begin{equation}
    \mathcal{L}_\alpha[c](s) = \int_0^\infty e^{-\frac{s t^\alpha}{\alpha}} c\, t^{\alpha - 1} dt = c \int_0^\infty e^{-\frac{s t^\alpha}{\alpha}} t^{\alpha - 1} dt.
    \end{equation}
    Using the change of variable $u = \frac{s t^\alpha}{\alpha}$, we get $t = \left( \frac{\alpha u}{s} \right)^{1/\alpha}$, and
    \begin{equation}
    dt = \left( \frac{\alpha}{s} \right)^{1/\alpha} \cdot \frac{1}{\alpha} u^{\frac{1}{\alpha} - 1} du.
    \end{equation}
    Substituting into the integral:
    \begin{equation}
    \mathcal{L}_\alpha[c](s) = c \int_0^\infty e^{-u} \left( \frac{\alpha u}{s} \right)^{\frac{\alpha - 1}{\alpha}} \left( \frac{\alpha}{s} \right)^{1/\alpha} \cdot \frac{1}{\alpha} u^{\frac{1}{\alpha} - 1} du.
    \end{equation}
    Simplifying:
    \begin{equation}
    \mathcal{L}_\alpha[c](s) = \frac{c}{s} \int_0^\infty e^{-u} du = \frac{c}{s}.
    \end{equation}

    \item For $f(t) = t^q$, where $q \in \mathbb{R}$:
    \begin{equation}
    \mathcal{L}_\alpha[t^q](s) = \int_0^\infty e^{-\frac{s t^\alpha}{\alpha}} t^q\, t^{\alpha - 1} dt = \int_0^\infty t^{q + \alpha - 1} e^{-\frac{s t^\alpha}{\alpha}} dt.
    \end{equation}
    Using the same change of variable as in item 1:
    \begin{equation}
    u = \frac{s t^\alpha}{\alpha}, \quad t = \left( \frac{\alpha u}{s} \right)^{1/\alpha}, \quad dt = \left( \frac{\alpha}{s} \right)^{1/\alpha} \cdot \frac{1}{\alpha} u^{\frac{1}{\alpha} - 1} du.
    \end{equation}
    Substituting into the integral:
    \begin{align}
    \mathcal{L}_\alpha[t^q](s) &= \int_0^\infty \left( \frac{\alpha u}{s} \right)^{\frac{q + \alpha - 1}{\alpha}} e^{-u} \left( \frac{\alpha}{s} \right)^{1/\alpha} \cdot \frac{1}{\alpha} u^{\frac{1}{\alpha} - 1} du \\
    &= \left( \frac{\alpha}{s} \right)^{\frac{q + \alpha}{\alpha}} \cdot \frac{1}{\alpha} \int_0^\infty u^{\frac{q}{\alpha}} e^{-u} du \\
    &= \alpha^{q/\alpha} s^{-(1 + q/\alpha)} \Gamma\left(1 + \frac{q}{\alpha}\right).
    \end{align}
\end{enumerate}
\end{proof}

\subsection{Conformable Laplace Transform of a Delayed Function}

If $f(t) = 0$ for $t < 0$, and $T > 0$, then:
\begin{equation}
    \mathcal{L}_\alpha\{f(t - T)\}(s) = e^{-s T^\alpha/\alpha} \cdot \mathcal{L}_\alpha\{f(t)\}(s).
\end{equation}

This mirrors the classical case, where temporal delays appear as exponential multiplicative factors in the Laplace domain.



\end{document}